\documentclass[11pt,twoside,a4paper]{article}
     
\usepackage{amssymb}
\usepackage{amsmath}
\usepackage{amsthm, color}

\allowdisplaybreaks

\usepackage{amsfonts}
\usepackage{bbm}
\usepackage{verbatim}
\usepackage{url}
\usepackage[T1]{fontenc}

\usepackage{hyperref}
\hypersetup{
    colorlinks=true,
    linkcolor=blue,
    citecolor=blue,
    filecolor=magenta,      
    urlcolor=cyan,
}

\usepackage[
backend=biber,
style=alphabetic,
maxbibnames=10,
sorting=nyt
]{biblatex}
\DeclareUnicodeCharacter{0301}{\"u}

\newcommand{\R}{\mathbb{R}}
\def\lz{\lambda}

\def\gs{\gtrsim}

\def\tbz{{\triangle_\lz}}

\newtheorem{thm}{Theorem}[section]
\newtheorem{lem}[thm]{Lemma}%[section]
\newtheorem{prop}[thm]{Proposition}%[section]
\newtheorem{defn}[thm]{Definition}%[section]
\numberwithin{equation}{section}

\begin{document}

\arraycolsep=1pt

\title{\Large\bf Muckenhoupt-type Weights and Quantitative Weighted Estimates in the Bessel Setting}
\author{Ji Li, Chong-Wei Liang,  Chun-Yen Shen, and Brett D. Wick}

%\medskip
\date{}
\maketitle

{\noindent  {\bf Abstract:}
Part of the intrinsic structure of singular integrals in the Bessel setting is captured by Muckenhoupt-type weights. Anderson--Kerman showed that the Bessel Riesz transform is bounded on weighted $L^p_w$ if and only if $w$ is in the class $A_{p,\lambda}$.    We introduce a new class of Muckenhoupt-type weights $\widetilde A_{p,\lambda}$  in the Bessel setting, which is different from $A_{p,\lambda}$ but characterizes the weighted boundedness for the Hardy--Littlewood maximal operators. We also establish the weighted $L^p$ boundedness and compactness, as well as the endpoint weak type boundedness of Riesz commutators. The quantitative weighted bound is also established.

\section{Introduction}
\subsection{Background}

In this paper, we develop new Muckenhoupt-type weights and establish the weighted boundedness characterization in the Bessel setting,
along the line of Muckenhoupt and Stein \cite{muckenhoupt1965classical}, Anderson--Kerman \cite{andersen1981weighted}, and our recent result \cite{MR4737319}.

Consider the Bessel operator $\tbz$ defined by Muckenhoupt and Stein in \cite{muckenhoupt1965classical}
\begin{equation*}
\tbz :=-\frac{d^2}{dx^2}-\frac{2\lz}{x}\frac{d}{dx},\quad \lz>0
\end{equation*}
%They developed a theory in the setting of
%$\tbz$ which parallels the classical one associated to the usual Laplacian $\triangle$.  Results on the $\lpz$-boundedness of conjugate
%functions and fractional integrals associated with $\tbz$ were
%obtained, where $p\in[1, \fz)$, $\rr_+:=(0, \fz)$ 
with  measure $d\mu(x):= x^{2\lz}\,dx$ and the Bessel Riesz transform $R_\lambda:=\frac{d}{dx}(\triangle_\lambda)^{-\frac{1}{2}}$.

As noted in \cite{MR4737319}, if we just simply treat $(\R_+,|\cdot|, d\mu)$  as a space of homogeneous type in the sense of Coifman and Weiss \cite{coifman1977extensions}, then it does not fully reflect the intrinsic structure in the Bessel setting.
For example, Anderson and Kerman \cite{andersen1981weighted} first introduced a Muckenhoupt-type weights $A_{p,\lambda}$ with respect to the Lebesgue measure (but not $dm_\lambda$), i.e., for $w\in A_{p,\lambda}$
\begin{align}\label{new weight Lp}
\|f\|_{L^p(\R_+,wdx)}:=\bigg(\int_0^\infty |f(x)|^pw(x)dx\bigg)^{1\over p },
\end{align}
and proved: 
``$R_\lambda$ is bounded on $L^p_w(\R_+)$ (in the sense of \eqref{new weight Lp}) {\bf if and only if }
$w$ is in $A_{p,\lambda}$''.

In \cite{MR4737319} we studied the natural question: does $A_{p,\lambda}$ also characterize  the weighted boundedness of Hardy--Littlewood maximal function in the Bessel setting $(\R_+,|\cdot|, d\mu)$? The answer is no! The Hardy--Littlewood maximal function is linked to another Muckenhoupt-type weight introduced in \cite{MR4737319}: $\widetilde A_{p,\lambda}$ (given in Definition \ref{tilde Ap}).

The aim of this paper is to study the weighted estimates with respect to $\widetilde A_{p,\lambda}$, including 

1. quantitative weighted norm for  $R_\lambda$; 

2. the weighted $L^p$ boundedness, compactness of the commutator $[b,R_\lambda]$ in terms of BMO, and VMO spaces, respectively; 

3. the weighted endpoint estimate for $[b,R_\lambda]$ with respect to $\widetilde A_{1,\lambda}$.

Before stating our main results, it is natural to ask, is there any difference between these weighted estimates and those obtained before in the classical $A_p(\mu)$ setting on $(\R_+,|\cdot|, d\mu)$? 

Recall that the weighted $L^p$ for the classical $A_p(\mu)$ setting on $(\R_+,|\cdot|, d\mu)$ is given by 
$$\|f\|_{L^p(\R_+,wd\mu)}:=
\bigg(\int_0^\infty |f(x)|^pw(x)d\mu(x)\bigg)^{1\over p }.$$

We first point out that $A_p(\mu)$ and $\widetilde A_{p,\lambda}$ are different. As classes of weights, they have a non-trivial intersection, but are distinct in that one does not contain the other, which is reflected by the indices of the power weight.
\begin{prop}\label{Apmu Aplambda}
    Let $w(t):=t^{\alpha}$ on $\R_+$. Then we have\\
(1) $w\in A_{p} (\mu)\Longleftrightarrow -1-2\lambda<\alpha<p-1+2\lambda(p-1)$,\\
(2) $w\in\widetilde{A}_{p,\lambda}\Longleftrightarrow -1<\alpha<p-1+(2\lambda+1) p.$
\end{prop}

\subsection{Statement of main results}
We first recall the definition of new Muckenhoupt weight class $\widetilde A_{p,\lambda}$ in \cite{MR4737319} (throughout this paper, we denote  
$d\nu_{\lambda}(x):=x^{2\lambda +1}dx$).
\begin{defn}\label{tilde Ap}
Suppose $1<p<\infty$.  For a non-negative measurable function $w$, $w\in \widetilde{A}_{p,\lambda}$ if there exists a constant $C>0$ such that for all interval $B\subset \R_+$ we have
  $$
\bigg(\frac{1}{\nu_\lambda(B)}\int_B w(t)dt\bigg)\bigg(\frac{1}{\nu_\lambda(B)}\int_B t^{
(2\lambda+1) p'}w(t)^{-\frac{1}{p-1}}dt\bigg)^{p-1}<C.
$$
For $p=1$,
$w\in \widetilde{A}_{1,\lambda}$ if there exists a constant $C>0$ such that for all interval $B\subset \R_+$ we have
$$
\frac{w(B)}{\nu_\lambda(B)}\leq C\frac{w(x)}{\nu_\lambda(x)},
$$  
for almost all $x\in B$.  For $p=\infty$, we define the borderline class $\widetilde{A}_{\infty,\lambda}$ to be
$
\widetilde{A}_{\infty,\lambda}:=\underset{p\geq1}{\bigcup}\widetilde{A}_{p,\lambda}.
$
%, where $\nu_\lambda(t):=t^{1+2\lambda}.$
\end{defn}
As noted in \cite{MR4737319}, $\widetilde{A}_{1,\lambda}$ is different from $A_{1,\lambda}$. $A_{1,\lambda}$ is not the limiting class of $A_{p,\lambda}$ (as showed in \cite{andersen1981weighted}), while  $\widetilde{A}_{1,\lambda}$ is the limiting class of $\widetilde{A}_{p,\lambda}$.

\subsubsection{Quantitative weighted estimates and characterisations of BMO space}

%In \cite{duong2021two}, the Bessel Riesz transform, $R_\lambda$, is treated as a Calder\'on--Zygmund operator on the space of homogeneous type $(\R_+,|\cdot|,\mu)$. Inspired by \cite{LERNER2019225}, it is natural to consider the sparse operator on $(\R_+,|\cdot|,\mu)$.
Let $\mathcal{D}$ be the system of dyadic cubes on  $(\R_+,|\cdot|,\mu)$, and $S\subset\mathcal{D}$ be the sparse family of cubes (details of definition are provided in Section 2). Along the line of \cite{a2bound} and \cite{LERNER2019225}, we consider the sparse operator on $(\R_+,|\cdot|,\mu)$.
\begin{defn}[\cite{duong2021two}]
Let $\Lambda\in (0,1)$ and $S$ be a $\Lambda$-sparse family of dyadic cubes. The sparse operator associated with the collection $S$ is defined by
    \[
\mathcal{A}_{\mathcal{S}}f(x):=\sum_{Q\in\mathcal{S}}\frac{1}{\mu(Q)}\int_{Q}f(t)\,d\mu(t)\cdot\chi_{Q}(x).
    \]
\end{defn}
In \cite{duong2021two}, they treated 
$(\R_+, |\cdot|, \mu)$ as a space of homogeneous type and obtained the quantitative estimate via the $A_p(\mu)$. As already pointed out in Proposition \ref{Apmu Aplambda}, $A_p(\mu)$ and our new weight class are different. 

Thus, the first main result is the quantitative  estimates for the the new class of weights.
\begin{thm}\label{thma}
  Let $1<p<\infty$. Then for $w\in\widetilde{A}_{p,\lambda-\frac{1}{2}} $, one has \[
\|\mathcal{A}_{\mathcal{S}}\|_{L^p(\R_+,wdx)\to L^p(\R_+,wdx)} \lesssim [w]^{\max\{1,\frac{1}{p-1}\}}_{\widetilde{A}_{p,\lambda-\frac{1}{2}}}.
  \]  
\end{thm}

Next, we establish the weighted boundedness and compactness characterizations for the commutator 
$[b,R_\lambda]$ with respect to $\widetilde A_{p,\lambda-{1\over2}}$. Note that the weighed boundedness of $[b, R_\lambda]$ with respect to $A_{p,\lambda}$ was first established in \cite{MR4737319} while the compactness with respect to $A_{p,\lambda}$ is still open.

%\begin{thm}
%From \cite{DLMWY} we see that ${\rm VMO}_{\triangle_\lambda} (\mathbb{R_+})$ is the closure of $C_0^\infty(\R_+)$ under the ${\rm BMO}_{\triangle_\lambda} (\mathbb{R_+})$ norm.
%\end{thm}

\begin{defn}[Sparse Operator associated to the BMO function]~\\
Let $\Lambda\in (0,1)$ and $S$ be a $\Lambda$-sparse family of dyadic cubes and $b\in BMO_{\triangle_\lambda}$. The sparse operator associated with the BMO function and the collection $S$ is defined by
    \[
\mathcal{A}_{\mathcal{S},b}f(x):=\sum_{Q\in\mathcal{S}}\frac{1}{\mu(Q)}|b(x)-b_Q|\int_{Q}f(t)\,d\mu(t)\cdot\chi_{Q}(x).
    \]
\end{defn}
The adjoint of $\mathcal{A}_{\mathcal{S},b}$ with respect to the underlying measure $\mu$ is $\mathcal{A}^*_{\mathcal{S},b}$, given by
 \[
\mathcal{A}^*_{\mathcal{S},b}f(x):=\sum_{Q\in\mathcal{S}}\frac{1}{\mu(Q)}\int_{Q}|b(t)-b_Q|f(t)\,d\mu(t)\cdot\chi_{Q}(x).
 \]

%Since 

%A fact relates to these two operator is that they dominates the commutator associated some operators in Bessel setting:
%\begin{thm}
%    Let $T$ be a Calder\'on--Zygmund operator on $(X,d,\mu)$ and $b\in L^1_{loc} (X)$. For $f\in L^\infty (X)$ with bounded support, there exists $\mathcal{T}$ dyadic systems $\mathcal{D}^{(t)}$ and a $\Lambda$-sparse families $S_t\subset \mathbf{D}^{(t)}$ such that for almost all $x\in X$
 %   \[
%|[b,T]f(x)|\lesssim\sum_{t}\bigg(\mathcal{A}_{\mathcal{S},b}(|f|)(x)+ \mathcal{A}^{*}_{\mathcal{S},b}(|f|)(x)\bigg).
 %   \]
%\end{thm}

The second main result is as follows.
\begin{thm}\label{thmb}
Let $1<p<\infty$. Then for $w\in\widetilde{A}_{p,\lambda-\frac{1}{2}} $, we have
\[
\|\mathcal{A}_{\mathcal{S},b}\|_{L^p(\R_+,wdx)\to L^p(\R_+,wdx)} ,\ \ \|\mathcal{A}^*_{\mathcal{S},b}\|_{L^p(\R_+,wdx)\to L^p(\R_+,wdx)} \lesssim \|b\|_{{\rm BMO}_{\triangle_\lambda}(\R_+)}[w]^{2\max\{1,\frac{1}{p-1}\}}_{\widetilde{A}_{p,\lambda-\frac{1}{2}}}.
  \]  
\end{thm}

%To get the weighted endpoint estimate for $[b, R_\lambda]$, in this paper, 
We also show the  equivalence of the BMO space with respect to the new class of weights.
\begin{defn}[Weighted $p-\rm BMO$ Space]~\\
Let $w$ be a non-negative, locally integrable function. We define for each $1\leq p<\infty$, the weighted $p-\rm BMO$ space ${\rm BMO}_{L^p(w)} (\mathbb{R_+})$ of bounded mean oscillation on $\mathbb{R_+}$ is the set of 
$f\in L^1_{loc} (\mathbb{R_+},wdx)$ such that
\begin{align*}
\|f\|_{{\rm BMO}_{L^p(w)}(\R_+)}:= \sup_{B\subset \mathbb R_+} \bigg(\frac{1}{w(B)}\int_B|f(t)-f_{B}|^p w(t)dt\bigg)^{\frac{1}{p}}<\infty,
\end{align*}
where
%\begin{align}\label{fBlambda}
$f_{B}:= \frac{1}{\mu(B)} \int_B f(t)\, d\mu(t)$, and $B=(a,b).$
%\end{align}
\end{defn}

Consider also the  space $\rm
 BMO_{0,s}(\R_+,w)$ (motivated by  \cite{f348cf10-eb78-3134-8834-84c7e5fe174a})
with the quantity defined by
\[
\|b\|_{\rm
 BMO_{0,s}(\R_+,w)}:=\underset{B}{\sup} \underset{c}{\inf}\inf\{t\geq 0:w(\{x\in B:|b(x)-c|>t\})\leq s\cdot w(B)\}.
\]
For $w\equiv 1$, we simply denote $\rm
 BMO_{0,s}(\R_+,w)$ by $\rm
 BMO_{0,s}(\R_+)$.

\begin{thm}\label{thm BMO equiv}
For all $w\in \widetilde{A}_{\infty,\lambda-\frac{1}{2}}$ and $0<s\leq\frac{1}{2}$, we have
\begin{align*}
    {\rm BMO}_{0,s} (\mathbb{R_+},w)&\simeq
    {\rm BMO}_{L^p(w)}(\R_+)
    \simeq{\rm BMO}_{\triangle_\lambda}(\R_+)
    \simeq\rm BMO(\R_+)
    \simeq
    {\rm BMO}_{L^p(dx)}(\R_+)\simeq\rm
 BMO_{0,s}(\R_+).
\end{align*}
%Note that the second $\simeq$ is deduced by the $A_2$ condition and the induction method in \cite{MR4737319}.
    
\end{thm}

\subsubsection{Applications to Bessel Riesz Transform and its Commutator}

As a consequence, we obtain the quantitative weighted estimates for  the Bessel Riesz transform $R_\lambda$ (as they are dominated by the suitable class of sparse operators \cite{duong2021two}).
\begin{thm}
     Let $1<p<\infty$. Then for $w\in\widetilde{A}_{p,\lambda-\frac{1}{2}} $, one has \[
 \|R_\lambda\|_{L^p(\R_+,wdx)\to L^p(\R_+,wdx)} \lesssim [w]^{\max\{1,\frac{1}{p-1}\}}_{\widetilde{A}_{p,\lambda-\frac{1}{2}}}.
  \]  
\end{thm}

Since the commutator $[b,R_\lambda]$ is dominated by a combination of $\mathcal{A}_{\mathcal{S},b}$ and $\mathcal{A}^*_{\mathcal{S},b}$, we obtain the weighted boundedness and compactness characterisation of $[b,R_\lambda]$, following the recent approach in \cite{chen2022compactness}. Here we only highlight the weighted compactness. We also make a remark here that the boundedness of commutators has a long history and close connections to other problems. For instance, one may look at the papers \cite{MR3231215}, \cite{MR3255002} and the references contained therein.  

\begin{thm}\label{thmc}
Suppose $b\in {\rm BMO}_{\triangle_\lambda} (\mathbb{R_+})$, $1<p<\infty$ and $w\in \widetilde{A}_{p,\lambda-\frac{1}{2}}$. Then the following statements hold.\\
(1) If $b\in {\rm VMO}_{\triangle_\lambda} (\mathbb{R_+})$, then  %$\mathcal{A}_{\mathcal{S},b}$ and $\mathcal{A}^*_{\mathcal{S},b}$ 
$[b,R_\lambda]$ is compact from $L^p(\R_+,wdx)$ to $L^p(\R_+,wdx)$;\\
(2) If $[b,R_\lambda]$
is compact from $L^p(\R_+,wdx)$ to $L^p(\R_+,wdx)$, then $b\in {\rm VMO}_{\triangle_\lambda} (\mathbb{R_+})$.
\end{thm}

With a full grasp of the $\rm BMO$ structure, the third main result is the following two theorems about the weighted endpoint characterization of $[b, R_\lambda]$.

\begin{thm}
    Let $b\in \rm BMO_{\triangle_\lambda}(\R_+)$ and $\phi$ be a Young function that satisfies
\[
C_\phi :=\int^\infty_1 \frac{\phi^{-1}(t)}{t^{2}\log(e+t)}dt<\infty.
\]
Then for all $t>0$ and a non-negative, locally integrable weight $w$, 
\[
 w(\{x\in\R_+:|[b,R_\lambda]f(x)|>t\})\lesssim C_\phi  \int_{\R_+} \Phi\left(2\|b\|_{\rm BMO_{\triangle_\lambda}(\R_+)}\cdot\frac{|f(x)|}{t}\right)M_{\Phi\circ\phi}\Big(\frac{w}{\mu}\Big)(x)d\mu(x),
\]
where $\Phi(x):=x\log(e+x).$
\end{thm}

%\begin{defn}[Borderline Class %$\widetilde{A}_{1,\lambda}$]~\\
%\end{defn}
 %In \cite{accoms}, the author shows that the weak type boundedness of the commutator would imply the bounded mean oscillation of $b$. 
 We then have the following weighted endpoint characterization with respect to $\widetilde{A}_{1,\lambda-\frac{1}{2}}$, where the unweighted prototype is obtained in \cite{perez1995endpoint} and \cite{accoms}.
\begin{thm}\label{thmiff}
Suppose $w\in \widetilde{A}_{1,\lambda-\frac{1}{2}}$. If $b\in \rm BMO_{\triangle_\lambda}(\R_+)$, then for all $t>0$
 \begin{align*}
 w(\{x\in\R_+:|[b,R_\lambda]f(x)|>t\})\underset{w,\Phi}{\lesssim} \int_{\R_+} \Phi\left(\|b\|_{\rm BMO_{\triangle_\lambda}(\R_+)}\frac{|f(x)|}{t}\right)w(x)dx,
 \end{align*}
where $\Phi(x):=x\log (x+e).$
Conversely, if there is some $C>0$ such that %the above inequality holds,
\begin{align*}
 w(\{x\in\R_+:|[b,R_\lambda]f(x)|>t\})\underset{w,\Phi}{\lesssim} \int_{\R_+} \Phi\left(C\frac{|f(x)|}{t}\right)w(x)dx,
 \end{align*}
 then
 $b\in \rm BMO_{\triangle_\lambda}(\R_+)$.
\end{thm}
%We remark that $\mu$ is a particular weight in $\widetilde{A}_{1,\lambda-\frac{1}{2}}$.
Finally, we present the fact that the $L\log L$-scale is the most reasonable scale when concerning  the endpoint estimate.
\begin{thm}\label{thmcounter}
There is $b_0$ in ${\rm BMO}_{\Delta_\lambda}(\R_+)$ such that    $[b_0,R_\lambda]$ is not of weak type $(1,1)$.
\end{thm}

%\newpage

\section{Preliminaries}

To begin with, we first recall the BMO and VMO spaces in \cite{ duong2017factorization}.
\subsection{BMO and VMO Space}
\begin{defn}[The space ${\rm BMO}_{\triangle_\lambda} (\mathbb{R_+})$]\label{defbmo}~\\
The BMO space ${\rm BMO}_{\triangle_\lambda} (\mathbb{R_+})$ of $\lambda$-bounded mean oscillation on $\mathbb{R_+}$ is the set of 
$f\in L^1_{loc} (\mathbb{R_+},\mu)$ such that
\begin{align}\label{fB}
\|f\|_{{\rm BMO}_{\triangle_\lambda}}:= \sup_{B\subset \mathbb R_+} \frac{1}{\mu(B)}\int_B|f(t)-f_{B}|d\mu(t)<\infty,
\end{align}
where
%\begin{align}\label{fBlambda}
$f_{B}:= \frac{1}{\mu(B)} \int_B f(t)\, d\mu(t)$, and $B=(a,b).$
%\end{align}
\end{defn}

\begin{defn}[The space ${\rm VMO}_{\triangle_\lambda} (\mathbb{R_+})$]\label{defvmo}~\\
The VMO space ${\rm VMO}_{\triangle_\lambda} (\mathbb{R_+})$ of $\lambda$-vanishing mean oscillation on $\mathbb{R_+}$ is the set of 
$f\in {\rm BMO}_{\triangle_\lambda} (\mathbb{R_+})$ such that
\begin{align}\label{limit 1}
&\lim_{r\to0}\ \sup_{\substack{B=(a,b)\subset \mathbb R_+\\ b-a=r}} \frac{1}{\mu(B)}\int_B|f(t)-f_{B}|d\mu(t)=0,\\
&\lim_{r\to\infty}\ \sup_{\substack{B=(a,b)\subset \mathbb R_+\\ b-a=r}} \frac{1}{\mu(B)}\int_B|f(t)-f_{B}|d\mu(t)=0,\\
&\lim_{a\to\infty}\ \sup_{B=(a,b)\subset \mathbb R_+} \frac{1}{\mu(B)}\int_B|f(t)-f_{B}|d\mu(t)=0.
\end{align}
\end{defn}

Apart from this, we also introduce the notion of the dyadic grid in this section.
\subsection{A System of Dyadic Cubes}\label{sec:dyadic_cubes}
In the space of homogeneous type $(X,d,\mu)$, a
countable family
$$
    \mathcal{D}
:=\underset{k\in\mathbb{Z}}{\bigcup}{D}_k, \
    \mathcal{D}
    :=\{Q^k_\alpha\colon \alpha\in \mathbb{A}_k\},
$$
of Borel sets $Q^k_\alpha\subseteq X$ is called \textit{a
system of dyadic cubes with parameters} $\delta\in (0,1)$ and
$0<a_1\leq A_1<\infty$ if it has the following properties:
\begin{itemize}
    \item 

    $X
    = \bigcup_{\alpha\in \mathbb{A}_k} Q^k_{\alpha}
    \quad\text{(disjoint union) for all}~k\in\mathbb{Z}$;

    \item$\text{if }\ell\geq k\text{, then either }
Q^{\ell}_{\beta}\subseteq Q^k_{\alpha}\text{ or }
        Q^k_{\alpha}\cap Q^{\ell}_{\beta}=\emptyset$;
\item 
    $\text{for each }(k,\alpha)\text{ and each } \ell\leq k,
    \text{ there exists a unique } \beta
    \text{ such that }Q^{k}_{\alpha}\subseteq Q^\ell_{\beta}$;
\item
    $\text{for each $(k,\alpha)$ there exists at most $M$
        (a fixed geometric constant)  $\beta$ such that }  $\\
    $Q^{k+1}_{\beta}\subseteq Q^k_{\alpha}$, $\text{ and }
        Q^k_\alpha =\bigcup_{{Q\in\mathbb{D}_{k+1},
    Q\subseteq Q^k_{\alpha}}}Q$;

    \item $B(x^k_{\alpha},a_1\delta^k)
    \subseteq Q^k_{\alpha}\subseteq B(x^k_{\alpha},A_1\delta^k)
    =: B(Q^k_{\alpha})$;
\item 
   $\text{if }\ell\geq k\text{ and }
   Q^{\ell}_{\beta}\subseteq Q^k_{\alpha}\text{, then }
   B(Q^{\ell}_{\beta})\subseteq B(Q^k_{\alpha})$.
\end{itemize} 
The set $Q^k_\alpha$ is called a \textit{dyadic cube of
generation} $k$ with centre point $x^k_\alpha\in Q^k_\alpha$
and sidelength~$\delta^k$.

%The interior and closure of
%$Q^k_\alpha$ are denoted by $\widetilde{Q}^k_{\alpha}$ and
%$\bar{Q}^k_{\alpha}$, respectively.
From the properties of the dyadic system above and from the doubling measure, we can deduce that there exists a constant
$C_{\mu,0}$ depending only on $C_\mu$, the doubling constant and $a_1, A_1$ as above, such that for any $Q^k_\alpha$ and $Q^{k+1}_\beta$  with $Q^{k+1}_\beta\subset Q^k_\alpha$,
\begin{align}\label{Cmu0}
\mu(Q^{k+1}_\beta)\leq \mu(Q^k_\alpha)\leq C_{\mu,0}\mu(Q^{k+1}_\beta).
\end{align}

We recall from \cite{hytonen2010systems} the following construction, which is a
slight elaboration of the seminal work by M.~Christ \cite{christ1990b}, as
well as Sawyer--Wheeden~\cite{2e1fcfce-3392-3df4-8cbf-5d7c1db26a90}.

\begin{thm}\label{theorem dyadic cubes}
On $(X,d,\mu)$, there exists a system of dyadic cubes with parameters
$0<\delta\leq (12A_0^3)^{-1}$ and $a_1:=(3A_0^2)^{-1},
A_1:=2A_0$. The construction only depends on some fixed set of
countably many centre points $x^k_\alpha$, having the
properties that
   $ d(x_{\alpha}^k,x_{\beta}^k)
        \geq \delta^k$ with $\alpha\neq\beta$,
    $\min_{\alpha}d(x,x^k_{\alpha})
        < \delta^k$ for all $x\in X,$
   and a certain partial order ``$\leq$'' among their index pairs
$(k,\alpha)$. In fact, this system can be constructed in such a
way that
$$%\begin{equation}\label{eq:closedCube}
    \overline{Q}^k_\alpha
    %\overline{{Q}^k_\alpha}
    =\overline{\{x^{\ell}_\beta:(\ell,\beta)\leq(k,\alpha)\}}, \quad\quad
%\end{equation}
%and
%\begin{equation}\label{eq:openCube}
    \widetilde{Q}^k_\alpha:=\operatorname{int}\overline{Q}^k_\alpha=
    \Big(\bigcup_{\gamma\neq\alpha}\overline{Q}^k_\gamma\Big)^c,
%\end{equation}
%$$
%and
%\begin{equation}\label{eq:3cubes}
\quad \quad   \widetilde{Q}^k_\alpha\subseteq Q^k_\alpha\subseteq %\overline{{Q}^k_\alpha},%
    \overline{Q}^k_\alpha,
%\end{equation}
$$
where $Q^k_\alpha$ are obtained from the closed sets
% $\overline{{Q}^k_\alpha}$
$\overline{Q}^k_\alpha$ and the open sets $\widetilde{Q}^k_\alpha$ by
finitely many set operations.
\end{thm}
We also recall the following remark from \cite{KAIREMA20161793}.
The construction of dyadic cubes requires their centre points
and an associated partial order be fixed \textit{a priori}.
However, if either the centre points or the partial order is
not given, their existence already follows from the
assumptions; any given system of points and partial order can
be used as a starting point. Moreover, if we are allowed to
choose the centre points for the cubes, the collection can be
chosen to satisfy the additional property that a fixed point
becomes a centre point at \textit{all levels}:
\label{fixedpoint}
\begin{itemize}
    \item 

    $\text{given a fixed point } x_0\in X, \text{ for every } k\in \mathbb{Z},
        \text{ there exists }\alpha \text{ such that }$ 
        
    $ x_0
        = x^k_\alpha,\text{ the centre point of }
Q^k_\alpha\in\mathcal{D}_k$.
\end{itemize}
The sparse family of dyadic cubes is similar to the setting in the case of Euclidean space.
\begin{defn}[Sparse Family of Dyadic Cubes]~\\
Given $0<\eta<1$, a collection $\mathcal S \subset \mathcal D$ of dyadic cubes is said to be $\eta$-sparse provided that for every $Q\in\mathcal S$, there is a measurable subset $E_Q \subset Q$ such that
$\mu(E_Q) \geq \eta \mu(Q)$ and the sets $\{E_Q\}_{Q\in\mathcal S}$ have only finite overlap.
\end{defn}
\section{Relation: Proof of Proposition \ref{Apmu Aplambda}}

Recall that for $w(t):=t^{\alpha}$, we have
\[
w\in\widetilde{A}_{p,\lambda}\Longleftrightarrow -1<\alpha<p-1+(2\lambda+1)p.
\]
The class that we focus on is $\widetilde{A}_{p,\lambda-\frac{1}{2}}$, and hence $\alpha$ must satisfy
\[
-1<\alpha<p-1+2\lambda p.
\]
%%Besides, we have the relation for all weight $w$
%%\[
%w\in\widetilde{A}_{p,\lambda-\frac{1}{2}}\Longleftrightarrow \frac{w(t)}{t^{2\lambda}}\in A_p (\mu),
%%\]
%where $\mu(t):=t^{2\lambda}.$
%Hence, in the case of power weight
%\begin{align*}
 %w \in A_p (\mu)&\Longleftrightarrow -1<\alpha+2\lambda<p-1+2\lambda p\\
 %&\Longleftrightarrow \textcolor{blue}{-1-%2\lambda<\alpha<p-1+2\lambda(p-1)}.
%\end{align*}
Now, we investigate the range of $\alpha$ such that $w\in A_p (\mu)$. Given a ball $B=B_{R}(x_{0})$, where $x_{0}$ is the center and $R$ is the radius, we consider the following two cases:
\begin{itemize}
    \item \textbf{Case (1)}:\,$|x_{0}|\geq3R$.
\end{itemize}
In the first case, for those $|x-x_{0}|<R$, we have $|x|\sim |x_{0}|$, hence
\[
\sup_{B}\bigg(\frac{1}{\mu(B)}\int_B |x|^{\alpha}d\mu\bigg)\bigg(\frac{1}{\mu(B)}\int_B |x|^{{-\alpha}\frac{p'}{p}}d\mu\bigg)^{p-1}\sim |x_{0}|^{\alpha}|x_{0}|^{-\alpha \frac{p'}{p}(p-1)} \sim 1.
\]
\begin{itemize}
    \item \textbf{Case (2)}:\,$|x_0|<3R$.
\end{itemize}
In the second case, one has
\[
B_{R}(x_0)\subset B_{4R}(0)\subset B_{7R}(x_0)
\]
and by the doubling property of $\mu$, we may replace $B_R(x_{0})$ by $B_{4R}(0)$ and hence compute the following
\begin{align*}
&\bigg(\frac{1}{R^{2\lambda+1}}\int_{-R}^{R}x^{\alpha}d\mu\bigg)\bigg(\frac{1}{R^{2\lambda+1}}\int_{-R}^{R}x^{-\alpha \frac{p'}{p}}d\mu\bigg)^{p-1}\\
&=\bigg(\frac{1}{R^{2\lambda+1}}\int_{-R}^{R}x^{\alpha+2\lambda}dx\bigg)\bigg(\frac{1}{R^{2\lambda+1}}\int_{-R}^{R}x^{-\alpha \frac{p'}{p}+2\lambda}dx\bigg)^{p-1}.
\end{align*}
If we want the above term to be finite uniformly in $R$, we only need to check both terms are integrable near $0$. Hence this gives the following conditions  for $\alpha$:
\[
\alpha+2\lambda >-1,\quad -\alpha\frac{p'}{p}+2\lambda>-1,
\]
which is equivalent to
\begin{align*}
 -1-2\lambda<\alpha < p-1+2\lambda(p-1).
\end{align*}
The proof is complete.

\section{Proof of Theorem \ref{thma}}
Define $\sigma(t):=w(t)^{1-p'}$ and $\sigma_*(t):=t^{2\lambda p'}w(t)^{1-p'}$.  Our strategy is to use the duality argument:
\[
\|\mathcal{A}_{\mathcal{S}}\|_{L^p(\R_+,wdx)\to L^p(\R_+wdx)}=\|\mathcal{A}_{\mathcal{S}}(\cdot\sigma)\|_{L^p(\R_+,\sigma dx)\to L^p(\R_+,wdx)}.
\]
\begin{itemize}
    \item Case $(1)$: $p\geq 2$
\end{itemize}
For all $f\in L^p(\R_+,\sigma dx)$ and $g\in L^{p'} (\R_+,wdx)$ with
$\|g\|_{L^{p'} (\R_+,wdx)}=1$,
\begin{align*}
 \langle\mathcal{A}_{\mathcal{S}}(f\sigma),g\rangle_{wdx}  &=\sum_{Q\in S} \frac{1}{\mu(Q)}\int_Q f(t) \sigma(t)d\mu(t)\cdot\int_Q g(x)w(x)dx\\
 &\leq  [w]_{\widetilde{A}_{p,\lambda-\frac{1}{2}}}\sum_{Q\in S} \frac{\mu(Q)^{p-1}}{w(Q)\sigma_* (Q)^{p-1}}\cdot\int_Q f(t) \sigma(t)d\mu(t)\cdot\int_Q g(x)w(x)dx\\
 &=  [w]_{\widetilde{A}_{p,\lambda-\frac{1}{2}}}\sum_{Q\in S} \frac{\mu(Q)^{p-1}}{\sigma_* (Q)^{p-1}}\cdot\int_Q f(t) \sigma(t)d\mu(t)\cdot\frac{1}{w(Q)}\int_Q g(x)w(x)dx\\
 &\lesssim [w]_{\widetilde{A}_{p,\lambda-\frac{1}{2}}}\sum_{Q\in S} \frac{\mu(E_Q)^{p-1}}{\sigma_* (E_Q)^{p-1}}\cdot\int_Q f(t) \sigma(t)d\mu(t)\cdot\frac{1}{w(Q)}\int_Q g(x)w(x)dx.
\end{align*}
Note that
\[
\mu(E_Q)^{p-1}\leq w(E_Q)^{\frac{p-1}{p}}\cdot \sigma_* (E_Q)^{\frac{p-1}{p'}},
\]
and hence
\begin{align*}&\langle\mathcal{A}_{\mathcal{S}}(f\sigma),g\rangle_{wdx}\\
&\lesssim [w]_{\widetilde{A}_{p,\lambda-\frac{1}{2}}}\sum_{Q\in S} \frac{\sigma_* (E_Q)^{\frac{1}{p}}}{\sigma_* (Q)}\cdot\int_Q f(t) \sigma(t)d\mu(t)\cdot\frac{w(E_Q)^{\frac{1}{p'}}}{w(Q)}\int_Q g(x)w(x)dx\\
&\leq [w]_{\widetilde{A}_{p,\lambda-\frac{1}{2}}}\bigg(\sum_{Q\in S} \frac{\sigma_* (E_Q)}{\sigma_* (Q)^p}\cdot(\int_Q f(t) \sigma(t)d\mu(t))^p\bigg)^{\frac{1}{p}}\cdot\bigg(\sum_{Q\in S}\frac{w(E_Q)}{w(Q)^{p'}}\cdot(\int_Q g(x)w(x)dx)^{p'}\bigg)^{\frac{1}{p'}}\\
&\leq [w]_{\widetilde{A}_{p,\lambda-\frac{1}{2}}}\cdot\|M^D_{\sigma_*} (f\cdot t^{2\lambda(1-p')})\|_{L^p (\R_+,\sigma_* dx)} \cdot \|M^D_w g\|_{L^{p'} (\R_+,wdx)}\\
&\leq [w]_{\widetilde{A}_{p,\lambda-\frac{1}{2}}}\cdot\|f\|_{L^p (\R_+,\sigma dx)} \cdot \|g\|_{L^{p'} (\R_+,wdx)},
\end{align*}
which implies
\[
\|\mathcal{A}_{\mathcal{S}}\|_{L^p(\R_+,wdx)\to L^p(\R_+,wdx)}\lesssim [w]_{\widetilde{A}_{p,\lambda-\frac{1}{2}}}.
\]
\begin{itemize}
    \item Case $(2):$ $1<p<2$
\end{itemize}
By Case $(1)$ and the relation
$$
\|\mathcal{A}_{\mathcal{S}}\|_{L^p(\R_+,wdx)\to L^p(\R_+,wdx)} = \|\mathcal{A}_{\mathcal{S}}\|_{L^{p'}(\R_+,\sigma_* dx)\to L^{p'}(\R_+,\sigma_* dx)}.
$$
Then we have
\[
\|\mathcal{A}_{\mathcal{S}}\|_{L^p(\R_+,wdx)\to L^p(\R_+,wdx)}\lesssim [\sigma_*]_{\widetilde{A}_{p',\lambda-\frac{1}{2}}}=[w]^{\frac{1}{p-1}}_{\widetilde{A}_{p,\lambda-\frac{1}{2}}}.
\] 
Hence, we get
\[
\|\mathcal{A}_{\mathcal{S}}\|_{L^p(\R_+,wdx)\to L^p(\R_+,wdx)} \lesssim [w]^{\max\{1,\frac{1}{p-1}\}}_{\widetilde{A}_{p,\lambda-\frac{1}{2}}},\quad\forall 1<p<\infty.
\]

\section{Proof of Theorem \ref{thmb}}
For all $f\in L^p(wdx)$ and $g\in L^{p'}(wdx)$ with $\|g\|_{L^{p'}(wdx)} =1$, we have
\begin{align*}
\langle\mathcal{A}_{\mathcal{S},b}f,g\rangle_{wdx}\leq \sum_{Q\in\mathcal{S}} \frac{1}{\mu(Q)}\int_Q |f(t)|d\mu(t)\cdot \int_Q |b(x)-b_Q||g(x)|w(x)dx.
\end{align*}
Note that there exists a sparse family of dyadic cubes $\tilde{\mathcal{S}}$ with $\mathcal{S}\subset\tilde{\mathcal{S}}$ such that
\[
|b(x)-b_Q|\cdot\chi_Q (x)\lesssim\sum_{P\in\tilde{\mathcal{S}}, P\subset Q } \frac{1}{\mu(P)}\int_P |b(x)-b_P|d\mu(t)\cdot\chi_P (x).
\]
Then one has
\begin{align*}
\langle\mathcal{A}_{\mathcal{S},b}f,g\rangle_{wdx}
&\lesssim\|b\|_{{\rm BMO}_{\triangle_\lambda}(\R_+)}\cdot\sum_{Q\in\mathcal{S}} \frac{1}{\mu(Q)}\int_Q |f(t)|d\mu(t)\cdot \sum_{P\in\tilde{\mathcal{S}}, P\subset Q }\mu(P)\cdot\bigg(\frac{1}{\mu(P)}\int_P |\tilde{g}(x)|d\mu(x)\bigg)\\
&\leq \|b\|_{{\rm BMO}_{\triangle_\lambda}(\R_+)}\cdot\sum_{Q\in\mathcal{S}} \frac{1}{\mu(Q)}\int_Q |f(t)|d\mu(t)\cdot \int_Q \mathcal{A}_{\tilde{\mathcal{S}}}|\tilde{g}|(x)d\mu(x)\\
&\leq \|b\|_{{\rm BMO}_{\triangle_\lambda}(\R_+)}\cdot\int_X \mathcal{A}_{\tilde{\mathcal{S}}}|f|(x)\mathcal{A}_{\tilde{\mathcal{S}}}|\tilde{g}|(x)d\mu(x),
\end{align*}
where $\tilde{g}(x):=g(x)w(x)x^{-2\lambda}.$
Besides, the sparse operator is a self-adjoint operator relative to the measure $\mu$, therefore
\begin{align*}
\langle\mathcal{A}_{\mathcal{S},b}f,g\rangle_{wdx}&\lesssim    \|b\|_{{\rm BMO}_{\triangle_\lambda}(\R_+)}\cdot\int_X \mathcal{A}_{\tilde{\mathcal{S}}}\bigg(\mathcal{A}_{\tilde{\mathcal{S}}}|f|(x)\bigg)|\tilde{g}|(x)d\mu(x)\\
&=\|b\|_{{\rm BMO}_{\triangle_\lambda}(\R_+)}\cdot\int_X \mathcal{A}_{\tilde{\mathcal{S}}}\bigg(\mathcal{A}_{\tilde{\mathcal{S}}}|f|(x)\bigg)|g|(x)w(x)dx\\
&\leq \|b\|_{{\rm BMO}_{\triangle_\lambda}(\R_+)}\cdot \|\mathcal{A}_{\tilde{\mathcal{S}}}(\mathcal{A}_{\tilde{\mathcal{S}}}|f|)\|_{L^p(\R_+,wdx)}\\
&\lesssim \|b\|_{{\rm BMO}_{\triangle_\lambda}(\R_+)}[w]^{2\max\{1,\frac{1}{p-1}\}}_{\widetilde{A}_{p,\lambda-\frac{1}{2}}}\cdot \|f\|_{L^p(\R_+,wdx)}.
\end{align*}
Hence,
\[
\|\mathcal{A}_{\mathcal{S},b}\|_{L^p(\R_+,wdx)\to L^p(\R_+,wdx)} \lesssim \|b\|_{{\rm BMO}_{\triangle_\lambda}(\R_+)}[w]^{2\max\{1,\frac{1}{p-1}\}}_{\widetilde{A}_{p,\lambda-\frac{1}{2}}}.
\]
Following the same argument, we have
\[
\|\mathcal{A}^*_{\mathcal{S},b}\|_{L^p(\R_+,wdx)\to L^p(\R_+,wdx)} \lesssim \|b\|_{{\rm BMO}_{\triangle_\lambda}(\R_+)}[w]^{2\max\{1,\frac{1}{p-1}\}}_{\widetilde{A}_{p,\lambda-\frac{1}{2}}}.
  \] 

  \section{Proof of Theorem \ref{thmc} }

\subsection{ Part $(1)$}
  We show $(1)$ first. 
  Consider the algorithm below. Let $H$ be some set in $\R$ and $f\in L^p(wdx)$, then 
  \begin{align*}
&|\mathcal{A}^*_{\mathcal{S},b}f(x)|\cdot\chi_{x\in H} (x)\\
&\leq \sum_{Q\in\mathcal{S}} \bigg(\frac{1}{\mu(Q)}\int_Q |b(y)-b_Q||f(y)|d\mu\bigg)\chi_Q (x)\cdot\chi_{x\in H} (x)\\
&\lesssim\sum_{Q\in\mathcal{S}}\sum_{R\in\tilde{\mathcal{S}},R\subset Q} \bigg(\frac{1}{\mu(R)}\int_R |b(y)-b_R|d\mu\bigg)\bigg(\frac{1}{\mu(Q)}\int_R |f(x)|d\mu(x)\bigg)\chi_Q (x)\cdot\chi_{x\in H} (x).
\end{align*}
Let $\varepsilon>0$. Suppose $b\in {\rm VMO}_{\triangle_\lambda} (\mathbb{R_+})$, then there exists a dyadic cube, $Q_A$, with $l(Q_A)=2^A$ such that for all dyadic cubes, $B$, with $B\cap Q_A=\emptyset$ or $Q_A\subset B$, one has
\[
\frac{1}{\mu(B)}\int_B|b(t)-b_{B}|d\mu(t)<\varepsilon.
\]
Moreovoer, there is $\delta>0$ such that \[
l(Q)<\delta\implies \frac{1}{\mu(B)}\int_B|b(t)-b_{B}|d\mu(t)<\varepsilon.
\]
Therefore, if $H:=\{|x|>K\}$, for some $K$ large enough, then
\begin{align*}
&|\mathcal{A}^*_{\mathcal{S},b}f(x)|\cdot\chi_{x\in H} (x)\\
&\lesssim\sum_{Q\in\mathcal{S},Q\cap B_K(0)=\emptyset}\sum_{R\in\tilde{\mathcal{S}},R\subset Q} \bigg(\frac{1}{\mu(R)}\int_R |b(y)-b_R|d\mu\bigg)\bigg(\frac{1}{\mu(Q)}\int_R |f(x)|d\mu(x)\bigg)\chi_Q (x)\\
&\leq\varepsilon\sum_{Q\in\mathcal{S},Q\cap B_K(0)=\emptyset}\sum_{R\in\tilde{\mathcal{S}},R\subset Q} \bigg(\frac{1}{\mu(Q)}\int_R |f(x)|d\mu(x)\bigg)\chi_Q (x)\\
&=\varepsilon\sum_{Q\in\mathcal{S},Q\cap B_K(0)=\emptyset}\frac{1}{\mu(Q)}\sum_{R\in\tilde{\mathcal{S}},R\subset Q} \mu(R)\bigg(\frac{1}{\mu(R)}\int_R |f(x)|d\mu(x)\bigg)\chi_Q (x)\\
&\leq\varepsilon\sum_{Q\in\mathcal{S},Q\cap B_K(0)=\emptyset}\frac{1}{\mu(Q)}\int_Q \mathcal{A}^*_{\tilde{\mathcal{S}}}|f|(x)d\mu(x)\chi_Q (x)\\
&\leq\varepsilon\mathcal{A}_{\mathcal{S}}(\mathcal{A}^*_{\tilde{{\mathcal{S}}}}|f|)(x)\chi_Q (x).
\end{align*}
By the same argument, if we consider the following decomposition
\begin{align*}
\mathcal{A}^*_{\mathcal{S},b}f(x)&=\sum_{Q\in\mathcal{S},Q\cap Q_A=\emptyset}\frac{1}{\mu(Q)}\int_{Q}|b(t)-b_Q|f(t)\,d\mu(t)\cdot\chi_{Q}(x)\\
&\quad+\sum_{Q\in\mathcal{S},Q_A\subset Q}\frac{1}{\mu(Q)}\int_{Q}|b(t)-b_Q|f(t)\,d\mu(t)\cdot\chi_{Q}(x)\\
&\quad+\sum_{Q\in\mathcal{S},Q\subset Q_A, l(Q)<\delta }\frac{1}{\mu(Q)}\int_{Q}|b(t)-b_Q|f(t)\,d\mu(t)\cdot\chi_{Q}(x)\\
&\quad+\sum_{Q\in\mathcal{S},Q\subset Q_A,l(Q)\geq\delta}\frac{1}{\mu(Q)}\int_{Q}|b(t)-b_Q|f(t)\,d\mu(t)\cdot\chi_{Q}(x).
\end{align*}
Then the first three terms are all bounded by \[
\varepsilon\,\mathcal{A}_{\mathcal{S}}(\mathcal{A}^*_{\tilde{{\mathcal{S}}}}|f|)(x)\chi_Q (x).
\]
Moreover, there is $0<N<\infty$ such that
\[
\#\{Q\in\mathcal{S}:Q\subset Q_A,l(Q)\geq\delta \}\leq N,
\]
then the fourth term
\[\sum_{Q\in\mathcal{S},Q\subset Q_A,l(Q)\geq\delta}\frac{1}{\mu(Q)}\int_{Q}|b(t)-b_Q|f(t)\,d\mu(t)\cdot\chi_{Q}(x)\] has finite range, and hence a compact operator.
Therefore,\[
b\in {\rm VMO}_{\triangle_\lambda} (\mathbb{R_+})\implies \mathcal{A}^*_{\mathcal{S},b}:L^p(\R_+,wdx) \underset{compact}{\longrightarrow} L^p(\R_+,wdx).
\]

\begin{itemize}
    \item  Compactness of $\mathcal{A}_{\mathcal{S},b}$:
\end{itemize}
Since
\[
\|\mathcal{A}_{\mathcal{S},b}\|_{L^p(\R_+,wdx)\to L^p(\R_+,wdx)}=
\|\mathcal{A}^*_{\mathcal{S},b}\|_{L^{p'}(\R_+,\sigma_* dx)\to L^{p'}(\R_+,\sigma_* dx)}<\infty
\]
and 
\[
w\in \widetilde{A}_{p,\lambda-\frac{1}{2}}\Longleftrightarrow
\sigma_*\in \widetilde{A}_{p',\lambda-\frac{1}{2}},
\]
 we have \[
b\in {\rm VMO}_{\triangle_\lambda} (\mathbb{R_+})\implies \mathcal{A}_{\mathcal{S},b}:L^p(\R_+,wdx) \underset{compact}{\longrightarrow} L^p(\R_+,wdx).
\]

\subsection{Part $(2)$} 

Now assume that $b\in {\rm BMO}_{\triangle_\lambda} (\mathbb{R_+})$ such that $\mathcal{A}_{\mathcal{S},b}$ and $\mathcal{A}^*_{\mathcal{S},b}$ are compact from $L^p(\R_+,wdx)$ to $L^p(\R_+,wdx)$. 
Hence, $[b,R_\lambda]$ is compact from $L^p(\R_+,wdx)$ to $L^p(\R_+,wdx)$.
But, for the sake of contradiction, further assume that   $b \not\in {\rm VMO}_{\triangle_\lambda} (\mathbb{R_+})$.  

We now use the idea from \cite{LACEY2022125869}: On a Hilbert space $\mathcal H$, with canonical basis $e_j$, $j\in \mathbb N$, 
an operator $T $ with $T e_j = v$, with non-zero $v\in \mathcal H$, is necessarily unbounded.

Suppose that $b\not\in {\rm VMO}_{\triangle_\lambda} (\mathbb{R_+})$, then at least one of the three conditions in Definition \ref{defvmo} does not hold.  
The argument is similar in all three cases, and we just present the case that the first condition in Definition~\ref{defvmo} does not hold.  
Then there exist $\delta_0>0$ and a sequence of balls $\{B_k\}_{k=1}^\infty=\{B_j(x_k,r_k)\}_{j=1}^\infty \subset \mathbb R_+$ such that $r_k\to 0$ as $j\to\infty$ and that
\begin{equation}
{1\over \mu(B_k)}\int_{B_k} |b(x)-b_{B_k}|d\mu(x)\geq\delta_0.
\end{equation}
Without loss of generality, we can further assume that 
\begin{align}\label{ratio1}
4 r_{k_{i+1}}\leq  r_{k_{i}}.
\end{align}

In Section 7.5 of \cite{duong2021two}, the authors proved the following argument holds for the kernel $K(x,y)$ of the operator $R_\lambda$:  There exist positive constants $3\le A_1\le A_2$ such that for any interval $B:=B(x_0, r)\subset \R_+$, there exists another interval $\widetilde B:=B(y_0, r)\subset \R_+$
such that $A_1 r\le |x_0- y_0|\le A_2 r$, %$\frac12B\subset \widehat B\subset B$, %$\mu(B)\sim \mu(\widetilde B)$
and  for all $(x,y)\in ( B\times \widetilde{B})$, $K(x, y)$ does not change sign and
\begin{equation}\label{e-assump cz ker low bdd}
|K(x, y)|\gs \frac1{\mu(\widetilde B)},
\end{equation}
where $d\mu(x):=x^{2\lambda}dx$ with $dx$ the Lebesgue measure on $\R_+$.

Let $\alpha_{b}(\widetilde B_k)$ be a median value of $b$ on the ball $\widetilde B_k $ with respect to $\mu$. Namely,  
$\alpha_{b}(\widetilde B_k)$ is a real number so that the two sets below have measures at least $ \frac{1}2 \mu  (\tilde B_j) $.  
\begin{align*}
\nonumber F_{k,1}\subset\{ y\in \tilde B_k: b(y)\leq \alpha_b(\tilde B_k) \},\qquad F_{k,2}\subset\{ y\in \tilde B_k: b(y)\geq \alpha_b(\tilde B_k) \}.
\end{align*}
Next we define
$E_{k,1}=\{ x\in B: b(x)\geq \alpha_b(\tilde B_k) \}$ and $E_{k,2}=\{ x\in B: b(x)< \alpha_b(\tilde B_k) \}.$\ 
Then $B_k=E_{k,1}\cup E_{k,2}$ and $E_{k,1}\cap E_{k,2}=\emptyset$. 
And it is clear that 
\begin{equation}\begin{split}\label{bx-by0-1 1}
b(x)-b(y) &\geq 0, \quad (x,y)\in E_{k,1}\times F_{k,1},\\ b(x)-b(y) &< 0, \quad (x,y)\in E_{k,2}\times F_{k,2}.
\end{split}\end{equation}
For $(x,y)$ in $(E_{k,1}\times F_{k,1} )\cup (E_{k,2}\times F_{k,2})$, we have 
\begin{align}\label{bx-by-1 1}
|b(x)-b(y)| 
&=|b(x)-\alpha_b(\tilde B_k)| +|\alpha_b(\tilde B_k) -b(y)| \geq |b(x)-\alpha_b(\tilde B_k)|. 
\end{align}

We now consider 
$$\widetilde F_{k,1}:= F_{k,1}\bigg\backslash \bigcup_{\ell=k+1}^\infty \tilde B_\ell\quad{\rm and}\quad \widetilde F_{k,2}:= F_{k,2}\bigg\backslash \bigcup_{\ell=k+1}^\infty \tilde B_\ell,\quad {\rm for}\ k=1,2,\ldots.$$
Then, based on the decay condition of the measures of  $\{B_j\}$ as in \eqref{ratio1} we obtain that
for each $j$,
\begin{align}\label{Fj1}
 \mu(\widetilde F_{k,1}) &\geq \mu(F_{k,1})- \mu\bigg( \bigcup_{\ell=k+1}^\infty \tilde B_\ell\Big)\\
 &\geq
{1\over 2} \mu(\tilde B_k)-\sum_{\ell=k+1}^\infty \mu(  \tilde B_\ell)\nonumber\\
&\geq {1\over 2} \mu(\tilde B_k)- {1\over 3}\mu(\tilde B_k)
= {1\over 6} \mu(\tilde B_k).\nonumber
\end{align}
Similar estimate holds for $\widetilde F_{k,2}$.

Now for each $j$, we have that
\begin{align*}
&{1\over \mu(B_k)} \int_{B_k} |b(x)-b_{B_k}|d\mu(x)\\
&\leq{2\over \mu(B_k)}\int_{B_{k}}\big|b(x)-\alpha_b(\tilde B_k)\big|d\mu(x)\\
&= {2\over \mu(B_k)}\int_{E_{k,1}}\big|b(x)-\alpha_b(\tilde B_k)\big|d\mu(x) + {2\over \mu(B_k)}\int_{E_{k,2}}\big|b(x)-\alpha_b(\tilde B_k)\big|d\mu(x).
\end{align*}
Thus, combining with the above inequalities, we obtain that as least one of the following inequalities holds:
\begin{align*}
{2\over \mu(B_k)}\int_{E_{k,1}}\big|b(x)-m_b(\tilde B_k)\big|d\mu(x) \geq {\delta_0\over2},\quad 
{2\over \mu(B_k)}\int_{E_{k,2}}\big|b(x)-m_b(\tilde B_k)\big|d\mu(x) \geq {\delta_0\over2}.
\end{align*}

Without lost of generality, we now assume that the first one holds, i.e., 
\begin{align*}
{2\over \mu(B_k)}\int_{E_{k,1}}\big|b(x)-m_b(\tilde B_k)\big|d\mu(x) \geq {\delta_0\over2}.
\end{align*}

Therefore, for each $j$, from \eqref{e-assump cz ker low bdd} and \eqref{Fj1}  we obtain that 
\begin{align*}
{\delta_0\over4}&\leq{1\over \mu(B_k)}\int_{E_{k,1}}\big|b(x)-\alpha_b(\tilde B_k)\big|d\mu(x)\\
&\lesssim 
{1\over \mu(B_k)}{{\mu(\widetilde F_{k,1})}\over \mu(B_k)}\int_{E_{k,1}}\big|b(x)-\alpha_b(\tilde B_k)\big|d\mu(x)\\
&\lesssim
{1\over \mu(B_k)}\int_{E_{k,1}}\int_{\widetilde F_{k,1}} |K(x,y)|\big|b(x)-b(y)\big|d\mu(y)d\mu(x).
\end{align*}
Next, since for $x\in E_{k,1}$ and $y\in \widetilde F_{k,1}$, $K(x,y)$ does not change sign and  
$b(x)-b(y)$ does not change sign either, we obtain that
\begin{align*}
{\delta_0}
&\lesssim {1\over \mu(B_k)}\bigg|\int_{E_{k,1}}\int_{\widetilde F_{k,1}} K(x,y)\big(b(x)-b(y)\big)d\mu(y)d\mu(x)\bigg|\\
&\lesssim
{1\over \mu(B_k)}  \int_{E_{k,1}}\left|[b, T](\chi_{\widetilde F_{{k,1}}})(x) \right|d\mu(x).
\end{align*}

Next, by using H\"older's
inequality we further have
\begin{align*}
\delta_0 &\lesssim {1\over    \mu(B_k)} \int_{E_{k,1}}\left|[b, T](\chi_{\widetilde F_{{k,1}}})(x)\right|  w^{1\over p}(x) w^{-{1\over p}}(x) d\mu(x)\\
&\lesssim
 {1\over \mu(B_k)} \sigma_* (E_{k,1})^{1\over p'} \bigg( \int_{\mathbb R_+}\big|[b, T](\chi_{\widetilde F_{{k,1}}})(x)\big|^p w(x) d\mu(x)\bigg)^{1\over p}\\
 &\leq\left(\frac{\sigma_*(E_{k,1})}{\sigma_*(B_k)}\right)^{1\over {p'}}\cdot w(B_k)^{\frac{-1}{p}}\bigg( \int_{\mathbb R_+}\big|[b, T](\chi_{\widetilde F_{{k,1}}})(x)\big|^p w(x) d\mu(x)\bigg)^{1\over p}\\
&=
\bigg( \int_{\mathbb R_+}\big|[b, T](f_k)(x)\big|^p w(x)d\mu(x) \bigg)^{1\over p},
\end{align*}
where in the above inequalities 
we denote
$$ f_k := {\chi_{\widetilde F_{k,1}}  \over w(B_k)^{1\over p}}.$$
This is a sequence of disjointly supported functions, by \eqref{Fj1}, with $\|f_k\|_{L^p(\R_+,wdx)} \sim 1 $.

Return to the assumption of compactness, and let $ \phi $ be in the closure of $  \{[b, T](f_j)\}_j$.  
We have $ \lVert \phi \rVert _{L ^{p} (\R_+,wdx )} \gtrsim 1$.   And, choose $ j_i$ so that 
$$ \|\phi- [b, T](f_{j_i})\|_{L^p(\R_+,wdx)} \leq  2^{-i}. $$

We then take non-negative numerical sequence $\{a_i\}$ with 
$$ \|\{a_i\}\|_{\ell^{p'}}<\infty \quad {\rm but}\quad   \|\{a_i\}\|_{\ell^1}=\infty.  $$  
Then, $ \psi =  \sum_{i} a_i f _{j_i} \in L ^{p} (\R_+,wdx ) $, and 
\begin{align*}
\Bigl\lVert  \sum_i {a_i} \phi   -  [ b, T ]{\psi } \Bigr\rVert_{L^p (\R_+,wdx )}  
&\leq \bigg\| \sum_{i=1}^\infty a_i   \bigl ( \phi - [b, T](f_{j_i}\bigr) \bigg\|_{L^p(\R_+,wdx)}
\\ & \leq \lVert  a_i\rVert _{\ell ^{p'}}  
\Bigl [   \sum_{i} \|\phi- [b, T](f_{j_i}) \| _{L^p(\R_+,wdx)} ^{p}  \Bigr] ^{1/p} \\
&\lesssim 1. 
\end{align*}
So $ \sum_i {a_i} \phi  \in L^p (\R_+,wdx )$.  
But $ \sum_i{a_i}\phi $ is infinite on a set of positive measures. This is a contradiction that completes the proof.

\section{Embedding: Proof of Theorem \ref{thm BMO equiv}}

To establish the comprehensive characterization of the weighted endpoint estimate of $[b, R_\lambda]$, we also discuss the embedding relation between some specific $\rm BMO$ space, as well.\\

For the class $\tilde{A}_{p,\lambda-\frac{1}{2}}$, the intertwining phenomenon between ${\rm BMO}_{L^p(w)} (\mathbb{R_+})$ and ${\rm BMO}_{\triangle_\lambda} (\mathbb{R_+})$ is the mutual embedding. A prototype in $\R^n$ for the classical $A_p$ weight was studied in \cite{ho2011characterizations}.

\begin{lem}\label{emb}
Let $0< p<\infty$ and $w\in \widetilde{A}_{\infty,\lambda-\frac{1}{2}}$, then
\[
{\rm BMO}_{\triangle_\lambda} (\mathbb{R_+})\hookrightarrow{\rm BMO}_{L^p(w)} (\mathbb{R_+}).
\]
 Conversely, let $1\leq p<\infty$ and if  $w\in \tilde{A}_{p,\lambda-\frac{1}{2}}$, then
\[
{\rm BMO}_{L^p(w)} (\mathbb{R_+})\hookrightarrow{\rm BMO}_{\triangle_\lambda} (\mathbb{R_+}).
\]
\end{lem}
Additionally, in \cite{f348cf10-eb78-3134-8834-84c7e5fe174a} the following  embedding holds
\begin{align}\label{sss}
  \forall 0<s\leq\frac{1}{2},\ \ p>0 \implies  s^{\frac{1}{p}}\cdot\|b\|_{\rm
 BMO_{0,s}(\R_+)} \leq\|b\|_{\rm
 BMO_{L^p(dx)}(\R_+)}\lesssim\|b\|_{\rm
 BMO_{0,s}(\R_+)}.
\end{align}

\begin{lem}\label{sec}
Let $0< p<\infty$ and $w$ be an arbitrary doubling measure, then for all $0<s\leq\frac{1}{2}$,
\[
{\rm BMO}_{0,s} (\mathbb{R_+ },w)\simeq{\rm BMO}_{L^p(w)} (\mathbb{R_+}).
\]    
More precisely,
\[
s^{\frac{1}{p}}\cdot\|b\|_{{\rm BMO}_{0,s} (\mathbb{R_+ },w)}\leq \|b\|_{{\rm BMO}_{L^p(w)} (\mathbb{R_+})}\underset{w,d}{\lesssim}\|b\|_{{\rm BMO}_{0,s} (\mathbb{R_+ },w)}
\]
\end{lem}

Combine the result in \eqref{sss}, Lemma \ref{emb}, Lemma \ref{sec} and the $\rm BMO$ equivalence lemma in \cite{MR4737319}, we can deduce theorem\ref{thm BMO equiv} directly. To be more specific,
\begin{align*}
    {\rm BMO}_{0,s} (\mathbb{R_+},w)&\underset{Lemma \,\, \ref{sec}}{\simeq}
    \rm BMO_{L^p(w)}(\R_+)
    \underset{Lemma\,\,\ref{emb}}{\simeq}\rm BMO_{\triangle_\lambda}(\R_+)\\
&\underset{Equivalence\,\, Lemma}{\simeq}\rm BMO(\R_+)
    \underset{Classical\,\, Result}{\simeq}
    \rm BMO_{L^p(dx)}(\R_+)
    \underset{(\ref{sss})}{\simeq}\rm
 BMO_{0,s}(\R_+).
\end{align*}

\begin{itemize}
    \item \textbf{Proof of Lemma \ref{emb}}
\end{itemize}
At the beginning of this section, we proved that if  $0< p<\infty$ and $w\in \tilde{A}_{\infty,\lambda-\frac{1}{2}}$, then
\[
{\rm BMO}_{\triangle_\lambda} (\mathbb{R_+})\hookrightarrow{\rm BMO}_{L^p(w)} (\mathbb{R_+}):
\]
By $w\in \tilde{A}_{\infty,\lambda-\frac{1}{2}}$ and  \cite[Corollary 6.6]{MR4737319}, there is $\delta>0$, such that for all $B\subset\R_+$
\begin{align*}
    \frac{w(\{x\in B: |b(x)-b_B|>\gamma\})}{w(B)}&\lesssim\left(\frac{\mu(\{x\in B: |b(x)-b_B|>\gamma\})}{\mu(B)}\right)^{\delta}\\
    &\hskip-1cm\underset{John-Nirenberg} {\lesssim}\exp{\left(\frac{-C\cdot\gamma\cdot\delta}{\|b\|_{\rm BMO_
    {\triangle_\lambda}}}\right)},
\end{align*}
which implies that 
\begin{align*}
   \frac{1}{w(B)}\int_B|b(t)-b_{B}|^p w(t)dt&=\frac{p}{w(B)}\int^\infty_0 \gamma^{p-1} w(\{x\in B: |b(x)-b_B|>\gamma\})d\gamma\\
   &\lesssim p\cdot\int^\infty_0\gamma^{p-1}\exp{\left(\frac{-C\cdot\gamma\cdot\delta}{\|b\|_{\rm BMO_{\triangle_\lambda}}}\right)}d\gamma\\
   &\lesssim \|b\|_{\rm BMO_{\triangle_\lambda}}^p.
\end{align*}
So,
\[
\|b\|_{{\rm BMO}_{L^p(w)}}\lesssim\|b\|_{\rm BMO_{\triangle_\lambda}}.
\]
Now, we show that
if  $1<p<\infty$ and $w\in \tilde{A}_{p,\lambda-\frac{1}{2}}$, then
\[
{\rm BMO}_{L^p(w)} (\mathbb{R_+})\hookrightarrow{\rm BMO}_{\triangle_\lambda} (\mathbb{R_+}):
\]
\begin{itemize}
    \item \textbf{Case (1):} $p>1$
\end{itemize}
By H\"older's inequality,
\begin{align*}
    \int_B |b(x)-b_B|d\mu(x)&\leq\bigg(\int_B|b(x)-b_B|^p w(x)dx \bigg)^{\frac{1}{p}}\cdot\bigg(\int_B w(x)^{1-p'}x^{2\lambda p'}dx\bigg)^{\frac{1}{p'}}\\
    &\hskip-.6cm\underset{w\in \tilde{A}_{p,\lambda-\frac{1}{2}}}{\lesssim}\bigg(\int_B|b(x)-b_B|^p w(x)dx \bigg)^{\frac{1}{p}}\cdot\frac{\mu(B)}{w(B)^{\frac{1}{p}}},
\end{align*}
and hence
\[
\|b\|_{\rm BMO_{\triangle_\lambda}}\lesssim\|b\|_{{\rm BMO}_{L^p(w)}}.
\]
\begin{itemize}
    \item \textbf{Case (2):} $p=1$
\end{itemize}
\begin{align*}
    \int_B |b(x)-b_B|d\mu(x)&=\int_B |b(x)-b_B|\cdot \frac{\mu(x)}{w(x)}w(x)dx\\
    &\leq\int_B|b(x)-b_B| w(x)dx \cdot\sup_{x\in B}\left(\frac{\mu(x)}{w(x)}\right)^{-1}\\
    &\hskip-.6cm\underset{w\in \tilde{A}_{1,\lambda-\frac{1}{2}}}{\lesssim}\int_B|b(x)-b_B| w(x)dx \cdot\frac{\mu(B)}{w(B)}.
\end{align*}
Therefore,
\[
\|b\|_{\rm BMO_{\triangle_\lambda}}\lesssim\|b\|_{{\rm BMO}_{L^1(w)}}.
\]
\begin{itemize}
    \item \textbf{Proof of Lemma \ref{sec}}
\end{itemize}
Now, we show that
for $0< p<\infty$ and $w\in \tilde{A}_{\infty,\lambda-\frac{1}{2}}$, then for all $0<s\leq\frac{1}{2}$
\[
{\rm BMO}_{0,s} (\mathbb{R_+ },w)\hookrightarrow{\rm BMO}_{L^p(w)} (\mathbb{R_+}):
\]    
Note that John--Str\"omberg's inequality in ${\rm BMO}_{0,s} (\mathbb{R_+ },w)$ is the following
\begin{align}\label{jsi}
w(\{x\in B: |b(x)-\alpha_b (B)|>\gamma\})\underset{w,d}{\lesssim}\exp{\left(\frac{-c\gamma}{\|b\|_{\rm BMO_{0,s}(\R_+,w)}}\right)}\cdot w(B),
\end{align}
where $\alpha_b (B)$ is the median value of $b$ on $B$, that is the number such that
$$
w(\{x\in B: b(x)>\alpha_b (B)\})\leq\frac{1}{2}w(B),
$$
and
$$
w(\{x\in B: b(x)<\alpha_b (B)\})\leq\frac{1}{2}w(B).
$$
In \cite{f348cf10-eb78-3134-8834-84c7e5fe174a} the author showed that John--Str\"omberg's inequality in the setting of ${\rm BMO}_{0,s} (\mathbb{R_+ })$. But if we consider $w$ to be doubling, then we would have the parallel John--Str\"omberg's inequality in ${\rm BMO}_{0,s} (\mathbb{R_+ },w)$. Their idea is based on the language of the  rearrangement. In the doubling measure setting, we can use the approach from \cite{f348cf10-eb78-3134-8834-84c7e5fe174a} and \cite{guo_wu_yang_2021} (Lemma $2.5$ and Proposition $3.2$).
\begin{defn}[Non-increasing Rearrangements for Doubling Measure]~\\
Given a real-valued measurable function $b$, the non-increasing rearrangements for doubling measure $w$ is defined by
$$
b^{*}(t):=\inf\{\gamma>0:w(x:\{|b(x)|>\gamma\})<t\},\quad\forall t\in\R_+ .
$$  
\end{defn}
\begin{defn}[Local Mean Oscillation for Doubling Measure]~\\
For a real-valued measurable function $b$, we defined the local mean oscillation for doubling measure $w$ of $b$ over a cube $B$ by
$$
\check{a}_{\lambda}(b;B):=\underset{c\in\R}{\inf}((b-c)\chi_B)^{*}(\lambda w(B)),\quad \forall \lambda\in(0,1).
$$
We also define
$$
a_{\lambda}(b;B):=\underset{c\in\R}{\inf}((f-\alpha_b (B))\chi_B)^{*}(\lambda w(B)),\quad \forall \lambda\in(0,1).
$$
\end{defn}

In the language of local mean oscillation, the median value plays the same role as the average in the language of the usual $\rm BMO$ norm. (\cite[Lemma 2.5]{guo_wu_yang_2021})
\begin{prop}\label{p1}
$$\check{a}_{\lambda}(b;B)
    \leq a_{\lambda}(b;B)\leq2 \check{a}_{\lambda}(b;B).$$
\end{prop}
The doubling condition gives us the estimate of the local mean oscillation. (\cite[Proposition 3.2]{guo_wu_yang_2021})
\begin{prop}\label{p2}
Let $0<\lambda\leq\frac{1}{2}$ and $B_{\varepsilon}\approx B$ be two balls with $w(B_{\varepsilon})=(1+\varepsilon)w(B)$ or   $w(B_{\varepsilon})=(1-\varepsilon)w(B)$. Then
$$
|\alpha_{b}(B_{\varepsilon})-\alpha_{b}(B)|\leq a_{\lambda}(b;B). 
$$
\end{prop}
With the aid of Propositions \ref{p1} and \ref{p2}, and follow the steps in \cite{f348cf10-eb78-3134-8834-84c7e5fe174a}(Section $3$), we have   \eqref{jsi}, John--Str\"omberg's inequality in ${\rm BMO}_{0,s} (\mathbb{R_+ },w)$. 

Back to the proof of $\rm BMO$ equivalence.  Let $b\in{\rm BMO}_{0,s}(\R_+,w) $, then we have
\begin{align*}
   \frac{1}{w(B)}\int_B|b(t)-\alpha_b (B)|^p w(t)dt&=\frac{p}{w(B)}\int^\infty_0 \gamma^{p-1} w(\{x\in B: |b(x)-\alpha_b (B)|>\gamma\})d\gamma\\
&\hskip-.1cm\underset{w,d}{\lesssim} p\cdot\int^\infty_0\gamma^{p-1}\exp{\left(\frac{-c\gamma}{\|b\|_{\rm BMO_{0,s}(\R_+,w)}}\right)}d\gamma\\
   &\lesssim \|b\|^p_{\rm BMO_{0,s}(\R_+,w)},
\end{align*}
that is $$
\|b\|_{{\rm BMO}_{L^p(w)} (\mathbb{R_+})}\underset{d}{\lesssim}\|b\|_{{\rm BMO}_{0,s} (\mathbb{R_+ },w)}.
$$
On the other hand,
\begin{align*}
w(\{x \in B:|b(x)-c|>s^{\frac{-1}{p}}\|b\|_{{\rm BMO}_{L^p(w)} (\mathbb{R_+})}\})&\leq s\cdot{\|b\|^{-p}_{{\rm BMO}_{L^p(w)} (\mathbb{R_+})}}\cdot\int_B |b(x)-c|^p w(x)dx\\
&\leq s\cdot w(B),
\end{align*}
which is equivalent to 
$$
s^{\frac{1}{p}}\cdot\|b\|_{{\rm BMO}_{0,s} (\mathbb{R_+ },w)}\leq \|b\|_{{\rm BMO}_{L^p(w)} (\mathbb{R_+})}.
$$

\section{Endpoint Estimate for Commutator $[b,R_\lambda]$}
In this section, we follow the idea from \cite{lerner} and \cite{10.1112/blms/bdv090}.   The Orlicz norm here is on the space of homogeneous type $(\R_+,d,\mu)$, where $\mu (t):=t^{2\lambda}$ is a doubling measure. Assume $\psi$ is a Young function such that
\[
\psi(4t)\leq\gamma_{\psi}\cdot\psi(t),\quad\forall t>0, \gamma_{\psi}\geq 1.
\]
Define the level set corresponding to the sparse family and the Young function $\psi$ by
\[
S_k:=\big\{Q\in \mathcal{S}: 4^{-k-1}<\|f\|_{\psi,Q}\leq 4
^{-k}\big\},\quad k\in \mathbb{N}.
\]
\begin{lem}\label{keylem}
Suppose $S_k$ is $(1-\frac{1}{2\gamma_{\psi}})$-sparse family of dyadic cubes. Let $w$ be a non-negative and locally integrable function and $E$ be such that $w(E)<\infty$, then for all Young function $\phi$,
\begin{align*}
    \int_E \sum_{Q\in S_k} w(x)dx\leq 2^k w(E)+\frac{4\gamma_\psi}{{\overline{\phi}^{-1}\big((2\gamma_\psi)^{2^k}\big)}}\cdot\int \psi(4^k|f|)(x) M_{\phi}\left(\frac{w}{\mu}\right)(x)d\mu(x).
\end{align*}
\end{lem}

\begin{proof}
We decompose the sparse family of cubes layer by layer.
\begin{itemize}
    \item \textbf{Layer Decomposition}:
\end{itemize}
Write $\mathcal{S}_k$ as the union of $\mathcal{S}_{k,v}$, for $v=0,1,\cdots$, where $\mathcal{S}_{k,0}$ are the maximal elements of $\mathcal{S}_k,$\,and $\mathcal{S}_{k,v+1}$ are the maximal elements of $\mathcal{S}_k\setminus\bigcup_{l=0}^v\mathcal{S}_{k,l}.$
\begin{itemize}
    \item \textbf{Pairwise Disjoint Subset}:
\end{itemize}
Define that
\[
E_Q:=Q\setminus\bigcup_{Q'\in\mathcal{S}_{k,v+1}}Q',\quad\forall Q\in\mathcal{S}_{k,v}.
\]
Note that $\{E_Q\}\overset{disjoint}{\subset}\mathcal{S}_k$.
Set $u:=2^k$ and for each $Q\in\mathcal{S}_{k,v}$, we decompose the set $E\cap Q$ into
\[
E\cap Q=E\cap\left(\underset{Bottom~ Part}{Q_u}\cup\underset{Layer~Part}{\bigcup_{l=0}^{u-1}\underset{{Q^\prime\in\mathcal{S}_{k,v+l},\subset Q}}{\bigcup} E_{Q^\prime}}\right),
\]
and hence
\begin{align}
\int_E \sum_{Q\in S_k} w(x)dx
&=\sum_{v=0}^{\infty}\sum_{Q\in\mathcal{S}_{k,v}} w(E\cap Q) \notag \\
&=\underset{Bottom~Part}{\sum_{v=0}^{\infty}\sum_{Q\in\mathcal{S}_{k,v}} w(E\cap Q_u)}
+\underset{Layer~Part}{\sum_{v=0}^{\infty}\sum_{Q\in\mathcal{S}_{k,v}}\sum_{l=0}^{u-1}\sum_{Q^\prime\in\mathcal{S}_{k,v+l},\subset Q} w(E_{Q'}\cap E)}.\notag
\end{align}
\begin{itemize}
    \item \textbf{Claim}:\[\forall Q\in S_{k,v}\implies 
    \frac{2\gamma_{\psi}}{\mu(Q)}\int_{E_Q}\psi(4^{k}|f|)(x)d\mu(x)\geq 1.
    \]
\end{itemize}
\begin{itemize}
    \item \textbf{Proof of Claim}:
    For each $Q\in S_{k,v}$, we have
    \[
    1<\frac{1}{\mu(Q)}\int_Q \psi(4^{k+1}|f(x)|)d\mu(x)\leq \frac{\gamma_\psi}{\mu(Q)}\int_Q \psi(4^{k}|f(x)|)d\mu(x).
    \]
    The sparseness implies that 
    \begin{align*}
      &\frac{1}{\mu(Q)}\int_Q \psi(4^{k}|f(x)|)d\mu(x)\\&=  \frac{1}{\mu(Q)}\int_{E_Q} \psi(4^{k}|f(x)|)d\mu(x)+\frac{1}{\mu(Q)}\sum_{Q'\in S_{k,v+1}}\int_{Q'} \psi(4^{k}|f(x)|)d\mu(x)\\
      &\leq\frac{1}{\mu(Q)}\int_{E_Q} \psi(4^{k}|f(x)|)d\mu(x)+\frac{1}{\mu(Q)}\sum_{Q'\in S_{k,v+1}}\mu(Q')\\
      &\leq\frac{1}{\mu(Q)}\int_{E_Q} \psi(4^{k}|f(x)|)d\mu(x)+\frac{1}{\mu(Q)}\mu(Q-E_Q)\\
      &\leq\frac{1}{\mu(Q)}\int_{E_Q} \psi(4^{k}|f(x)|)d\mu(x)+\frac{1}{2\gamma_\psi}.
    \end{align*}
    Thus, the claim holds.
\end{itemize}
\begin{itemize}
    \item \textbf{Bottom Part}:
\end{itemize}
\begin{align*}
\sum_{v=0}^{\infty}\sum_{Q\in\mathcal{S}_{k,v}} w(E\cap Q_u)
&\underset{Claim}{\leq}\sum_{v=0}^{\infty}\sum_{Q\in\mathcal{S}_{k,v}}w(E\cap Q_u)\cdot \frac{2\gamma_{\psi}}{\mu(Q)}\int_{E_Q}\psi(4^{k}|f|)(x)d\mu(x) \\
&\ \ =\sum_{v=0}^{\infty}\sum_{Q\in\mathcal{S}_{k,v}}2\gamma_{\psi}\cdot \frac{w(E\cap Q_u)}{\mu(Q)}\int_{E_Q}\psi(4^{k}|f|)(x)d\mu(x)\\
&\underset{Holder}{\leq}\sum_{v=0}^{\infty}\sum_{Q\in\mathcal{S}_{k,v}}4\gamma_{\psi}\left\|\frac{w}{\mu}\right\|_{\phi,Q}\cdot\left\|\chi_{Q_u}\right\|_{\overline{\phi},Q}\int_{E_Q}\psi(4^{k}|f|)(x)d\mu(x)\\
&\ \ =\sum_{v=0}^{\infty}\sum_{Q\in\mathcal{S}_{k,v}}4\gamma_{\psi}\left\|\frac{w}{\mu}\right\|_{\phi,Q}\frac{1}{\overline{\phi}^{-1}\left(\frac{\mu(Q)}{\mu(Q_u)}\right)}\int_{E_Q}\psi(4^{k}|f|)(x)d\mu(x)\\
&\hskip-.35cm\underset{sparseness}{\leq}\sum_{v=0}^{\infty}\sum_{Q\in\mathcal{S}_{k,v}}\frac{4\gamma_{\psi}}{{\overline{\phi}^{-1}((2\gamma_\psi)^{2^k})}}\int_{E_Q}\psi(4^{k}|f|)(x)M_{\phi}\left(\frac{w}{\mu}\right)(x)d\mu(x)\\
&\ \ \leq \frac{4\gamma_\psi}{{\overline{\phi}^{-1}((2\gamma_\psi)^{2^k})}}\cdot\int \psi(4^k|f|)(x) M_{\phi}\left(\frac{w}{\mu}\right)(x)d\mu(x).
\end{align*}
\begin{itemize}
    \item \textbf{Layer Part}:
\end{itemize}
\begin{align}
\sum_{v=0}^{\infty}\sum_{Q\in\mathcal{S}_{k,v}}\sum_{l=0}^{u-1}\sum_{Q^\prime\in\mathcal{S}_{k,v+l},\subset Q} w(E_{Q'}\cap E)
 &\leq u\sum_{v=0}^{\infty}\sum_{Q\in\mathcal{S}_{k,v}}w(E_{Q}\cap E)\notag\\
 &\leq 2^{k}w(E).\notag
\end{align}
Thus, we complete this lemma.
\end{proof}
\begin{lem}\label{Lemmalog}
Suppose $S$ is $\frac{31}{32}$-sparse family of dyadic cubes. Let $\phi$ be a Young function such that
\begin{align*}
    K_\phi :=\sum^{\infty}_{k=1}\frac{k}{{\overline{\phi}^{-1}((32)^{2^k})}}<\infty.
\end{align*}
Then for all $t>0$,
\begin{align*}
    w(\{x\in\R_+:|\mathcal{A}_{\mathcal{S},L\log L}(f)(x)|>t\})\lesssim K_\phi \cdot\int_{\R_+}\Phi\bigg(\frac{|f(x)|}{t}\bigg)M_{\phi}\left(\frac{w}{\mu}\right)(x)d\mu(x),
\end{align*}
where $\Phi:=x\log(e+x).$
\end{lem}
\begin{proof}
In this proof, we take the Young function $\psi(x):=x\log (e+x)$, the constant $\gamma_\psi =16$, and apply Lemma \ref{keylem}. Define the set 
\[
E:=\left\{4<\mathcal{A}_{\mathcal{S},L\log L}(f)\leq 8, M_{\mu,{L\log L}}(f)\leq\frac{1}{4}\right\}.    
    \]
    By homogeneity, we only need to show that
    \begin{align*}
        w(E)\lesssim K_\phi \int \psi(|f(x)|)M_\phi w (x)dx.
    \end{align*}
By Chebyshev's inequality, we have
\begin{align*}
w(E)&\leq\frac{1}{4}\int_E\mathcal{A}_{\mathcal{S},L\log L}f(x)w(x)dx\\
&=\frac{1}{4}\sum^{\infty}_{k=1}\int_E\mathcal{A}_{\mathcal{S}_k,L\log L}f(x)w(x)dx\\
&=\frac{1}{4}\sum^{\infty}_{k=1}\int_E\sum_{Q\in\mathcal{S}_k}\|f\|_{L\log L,Q}\chi_Q (x)w(x)dx\\
&\leq \frac{1}{4}\sum^{\infty}_{k=1} \frac{1}{4^k}\cdot\bigg( 2^k w(E)+\frac{64}{{\overline{\phi}^{-1}((32)^{2^k})}}\cdot\int \psi(4^k|f|)(x) M_{\phi}\left(\frac{w}{\mu}\right)(x)d\mu(x)\bigg)\\
&\lesssim \frac{1}{4}\sum^{\infty}_{k=1} \frac{1}{4^k}\cdot\bigg( 2^k w(E)+\frac{64}{{\overline{\phi}^{-1}((32)^{2^k})}}\cdot\int k 4^{k}\psi(|f|)(x) M_{\phi}\left(\frac{w}{\mu}\right)(x)d\mu(x)\bigg),
\end{align*}
which implies that
\begin{align*}
w(E)&\lesssim\sum^{\infty}_{k=1}\frac{k}{{\overline{\phi}^{-1}((32)^{2^k})}}\int \psi(|f|)(x)M_{\phi}\left(\frac{w}{\mu}\right)(x)d\mu(x).
\end{align*}

The proof is complete.
\end{proof}
\begin{itemize}
    \item \textbf{Young Functions, Luxemburg Norms, and John-Nirenberg Inequality}: Recall the fact that given any $b\in \rm BMO_{\triangle_\lambda}(\R_+)$
      There is a constant $C_s<+\infty$ such that for any interval $B$, 
    \[\hskip-6cm
    |s|<(ce\|f\|_{{\rm BMO}_{\triangle_\lambda}(\R_+)}^{-1}) \implies
    \]
    \[\frac{1}{\mu (B)}\int_{B}e^{s|b(x)-b_{B}|}d\mu(x) \leq     C_s:=1+e\cdot \frac{A^{-1}\cdot|s|\cdot\|b\|_{{\rm BMO}_{\triangle_\lambda}(\R_+)}}{1-A^{-1}\cdot|s|\cdot\|b\|_{{\rm BMO}_{\triangle_\lambda}(\R_+)}}.
    \]
  From this fact, if we take $\varphi(x):=e^x -1$, then we have \[
    \|b-b_B\|_{\varphi,B}\lesssim \|b\|_{\rm BMO_{\triangle_\lambda}(\R_+)},
    \]
    which can be deduced from 
    \[
\|f\|_{\varphi,B}\leq 1\Longleftrightarrow \frac{1}{\mu(B)}\int_B \varphi(|f(x)|)d\mu(x)\leq1 .
    \]
\end{itemize}
\begin{lem}
    Let $A$,$B$, and $C$ be non-negative, continuous, strictly increasing function on $[0\infty)$ with
    \[A^{-1} (t)\cdot B^{-1} (t)\leq C^{-1} (t),\quad\forall t\geq 0.\]
    Suppose that $C$ is also convex, then
    \[
    \|fg\|_{C,Q}\leq 2\|f\|_{A,Q}\|g\|_{B,Q}.
    \]
\end{lem}

\begin{lem}\label{change}
Let $\Psi$ be a Young function and $T$ be an operator such that for every Young function $\phi$,
\[
w(\{|Tf|>t\})\leq \int^\infty_1 \frac{\phi^{-1}(t)}{t^2}dt\cdot\int \Psi\Big(\frac{|f(x)|}{t}\Big)M_{\phi(L)}w(x)dx,\quad \forall t>0.
\]
Then
\[
w(\{|Tf|>t\})\lesssim\int^\infty_1 \frac{\phi^{-1}(t)}{t^{2}\log(e+t)}dt\int\Psi\Big(\frac{|f(x)|}{t}\Big)M_{\Psi\circ\phi(L)}w(x)dx,\quad\forall t>0.
\]
\end{lem}
Replace the function $\phi$ by $\Psi\circ\phi$, and then we have the desired one. Next, we use this lemma to prove the endpoint estimate for the sparse type operator associated with the commutator.
\begin{lem}
Let $b\in \rm BMO_{\triangle_\lambda}(\R_+)$  and $\mathcal{S}$ be $\frac{7}{8}$-sparse. Suppose $\phi$ is a Young function that satisfies
\[
C_\phi :=\int^\infty_1 \frac{\phi^{-1}(t)}{t^{2}\log(e+t)}dt<\infty.
\]
Then for all $t>0$,
\begin{align*}
 w(\{|\mathcal{A}_{\mathcal{S},b}(f)|>t\})\lesssim C_\phi\frac{ \|b\|_{{\rm BMO}_{\triangle_\lambda}(\R_+)}}{t} \int |f(x)|\cdot M_{\Phi\circ\phi}\left(\frac{w}{\mu}\right)(x)d\mu(x),
\end{align*}
in which $\Phi(x):=x\log (e+x).$
\end{lem}
\begin{proof}
In this proof, we take the Young function $\psi(x):=x$ the constant $\gamma_\psi =4$, and again apply Lemma \ref{keylem}. Define the set 
\[
E:=\left\{8<\mathcal{A}_{\mathcal{S},b}(f), M_{\mu}(f)\leq\frac{1}{4}\right\}.   
    \]
    By homogeneity,we are free to assume that $\|b\|_{{\rm BMO}_{\triangle_\lambda}(\R_+)}=1$, and hence only need to show that
    \begin{align*}
        w(E)\lesssim C_\phi \int |f(x)| M_{\Phi\circ\phi}\left(\frac{w}{\mu}\right)(x)d\mu(x).
    \end{align*}
    For $Q\in S_k$, let\[
    F_k:=\Big\{x\in Q: |b(x)-b_Q|\geq \left(\frac{3}{2}\right)^k\Big\},\quad k\in\mathbb{N}.
    \]
    Then we have
    \begin{align*}
   |\mathcal{A}_{\mathcal{S},b}(f)|(x)
&\leq\sum^\infty_{k=1}\sum_{Q\in S_k}\left(\frac{3}{2}\right)^k\frac{1}{\mu(Q)}\int_Q |f|d\mu\cdot\chi_Q (x)+ \sum^\infty_{k=1}\sum_{Q\in S_k}|b(x)-b_Q|\frac{1}{\mu(Q)}\int_Q |f|d\mu\cdot\chi_{F_k} (x)\\
   &=:\mathcal{A}^{(1)}_{\mathcal{S},b}(f)(x)+\mathcal{A}^{(2)}_{\mathcal{S},b}(f)(x).
    \end{align*}
Then
\begin{align*}
w(E)&\leq w(\{x\in E:\mathcal{A}^{(1)}_{\mathcal{S},b}(f)(x)>4\})+w(\{x\in E:\mathcal{A}^{(2)}_{\mathcal{S},b}(f)(x)>4\})\\
&=:w(E_1)+w(E_2).
\end{align*}
For the first term $w(E_1)$, one has
\begin{align*}
w(E_1)&\leq\frac{1}{4}\sum^\infty_{k=1}\sum_{Q\in S_k} \left(\frac{3}{2}\right)^k \cdot\frac{1}{\mu(Q)}\int_Q |f|d\mu\int_{E_1} \chi_Q w(x)dx\\
&\leq\frac{1}{4}\sum^\infty_{k=1} \left(\frac{3}{2}\right)^k\cdot \left(\frac{1}{4}\right)^k\cdot\bigg(2^k w(E)+\frac{16}{{\overline{\phi}^{-1}((8)^{2^k})}}\cdot\int \psi(4^k|f|)(x) M_{\phi}\left(\frac{w}{\mu}\right)(x)d\mu(x)\bigg)\\
&\leq\frac{1}{4}\sum^\infty_{k=1}\left(\frac{3}{4}\right)^k \cdot w(E_1)+\frac{1}{4}\sum^\infty_{k=1} \left(\frac{3}{2}\right)^k \frac{16}{{\overline{\phi}^{-1}((8)^{2^k})}}\cdot\int |f|(x) M_{\phi}\left(\frac{w}{\mu}\right)(x)d\mu(x),
\end{align*}
which implies that
\begin{align*}
w(E_1)&\leq\sum^\infty_{k=1} \left(\frac{3}{2}\right)^k \frac{16}{{\overline{\phi}^{-1}((8)^{2^k})}}\cdot\int |f|(x) M_{\phi}\left(\frac{w}{\mu}\right)(x)d\mu(x)\\
&\lesssim\int^\infty_1 \frac{\phi^{-1}(t)}{t^2}dt\cdot\int |f|(x) M_{\phi}\left(\frac{w}{\mu}\right)(x)d\mu(x).
\end{align*}
Here, we point out that the last inequality is by \[\overline{\phi}^{-1}(t)\cdot\phi^{-1}(t)\sim t\] and
\[
\sum^\infty_{k=1} \left(\frac{3}{2}\right)^k \frac{16}{{\overline{\phi}^{-1}((8)^{2^k})}}\lesssim\sum^\infty_{k=1}\int^{8^{2^{k}}}_{8^{2^{k-1}}}\frac{1}{\overline{\phi}^{-1}(t)}\frac{dt}{t}\lesssim\int^\infty_1 \frac{\phi^{-1}(t)}{t^2}dt.
\]
Thus, by Lemma \ref{change},
\[
w(E_1)\leq C_\phi\cdot\int |f(x)| M_{\Phi\circ\phi}\left(\frac{w}{\mu}\right)(x)d\mu(x).
\]
Similar to the previous lemma, we also have the quantity
\[
\frac{1}{\mu(Q)}\int_Q |f|d\mu\leq\frac{8}{\mu(Q)}\int_{E_Q} |f|d\mu,
\]
and hence for the second term, 
\begin{align*}
w(E_2)&\lesssim \sum^\infty_{k=1}\sum_{Q\in S_k}\frac{1}{\mu(Q)}\int_Q |f|d\mu\int_{F_k}|b(x)-b_Q|\cdot\frac{w(x)}{\mu(x)}d\mu(x)\\
&\lesssim\sum^\infty_{k=1}\sum_{Q\in S_k}\int_{E_Q} |f|d\mu\cdot\left\|b-b_Q\right\|_{expL,Q}\cdot \left\|\frac{w}{\mu}\chi_{F_K}\right\|_{L\log L,Q}\\
&\lesssim\sum^\infty_{k=1}\sum_{Q\in S_k}\int_{E_Q} |f|d\mu\cdot\left\|\frac{w}{\mu}\chi_{F_K}\right\|_{L\log L,Q}.
\end{align*}
To control the term $\|\frac{w}{\mu}\chi_{F_K}\|_{L\log L,Q}$, we consider the function
\[C(x):=x\log (e+x),\quad B(x):=\Phi\circ\phi(x)\]
and define the function $A$ to be with
\[
A^{-1}(x):=\frac{C^{-1}(x)}{B^{-1}(x)}.
\]
Thus,
\begin{align*}
 \left\|\frac{w}{\mu}\chi_{F_K}\right\|_{L\log L,Q}
 &\leq 2 \left\|\chi_{F_K}\right\|_{A,Q}\cdot\left\|\frac{w}{\mu}\right\|_{B,Q}
 =\frac{2}{A^{-1}\left(\frac{\mu(Q)}{\mu(F_k)}\right)}\cdot \left\|\frac{w}{\mu}\right\|_{B,Q}.
\end{align*}
Moreover, by John--Nirenberg Inequality, we have
\[
\mu(F_k)\leq \mu (Q)\cdot \min\{1,e^{1+\frac{-(\frac{3}{2})^k}{c\cdot e}}
\}.
\]
Therefore, if we define $\alpha_k:=\min\{1,e^{1+\frac{-(\frac{3}{2})^k}{c\cdot e}}
\}$, then
\begin{align*}
  w(E_2)  &\lesssim\sum^\infty_{k=1}\sum_{Q\in S_k}\int_{E_Q} |f|d\mu\cdot\frac{2}{A^{-1}(\frac{1}{\alpha_k})}\cdot \left\|\frac{w}{\mu}\right\|_{B,Q}\\
  &\leq \sum^\infty_{k=1}\frac{2}{A^{-1}(\frac{1}{\alpha_k})}\int |f|\cdot  M_{\Phi\circ\phi}\left(\frac{w}{\mu}\right)d\mu\\
  &\lesssim C_{\phi}\cdot \int |f|\cdot  M_{\Phi\circ\phi}\left(\frac{w}{\mu}\right)d\mu.
\end{align*}
The proof is complete.
\end{proof}
\begin{thm}
    Let $b\in \rm BMO_{\triangle_\lambda}(\R_+)$ and $\phi$ be a Young function that satisfies
\[
C_\phi :=\int^\infty_1 \frac{\phi^{-1}(t)}{t^{2}\log(e+t)}dt<\infty.
\]
Then for all $t>0$, 
\[
 w(\{|[b,R_\lambda]f|>t\})\lesssim C_\phi \cdot  \int \Phi\bigg(\|b\|_{\rm BMO_{\triangle_\lambda}(\R_+)}\cdot\frac{|f(x)|}{t}\bigg)M_{\Phi\circ\phi}\left(\frac{w}{\mu}\right)(x)d\mu(x),
\]
where $\Phi(x):=x\log(e+x).$
\end{thm}
\begin{proof}
By Lemma \ref{Lemmalog}, for all $t>0$,\[
w(\{|\mathcal{A}_{\mathcal{S},L\log L}(f)|>t\})\lesssim K_\phi \int \Phi\bigg(\frac{|f(x)|}{t}\bigg)M_{\phi}\left(\frac{w}{\mu}\right)(x)d\mu(x).\]
Hence, by lemma \ref{change}, we have 
\[w(\{|\mathcal{A}_{\mathcal{S},L\log L}(f)|>t\})\lesssim C_\phi \int \Phi\bigg(\frac{|f(x)|}{t}\bigg)M_{\Phi\circ\phi}\left(\frac{w}{\mu}\right)(x)d\mu(x),
\]
therefore combining all the lemmas we obtain that 
    \begin{align*}
        &w(\{|[b,R_\lambda]f|>t\})\\
        &\leq w(\{|\mathcal{A}^*_{\mathcal{S},b}(f)|>\frac{t}{2}\})+w(\{|\mathcal{A}_{\mathcal{S},b}(f)|>\frac{t}{2}\})\\
&\lesssim w(\{|\mathcal{A}_{\mathcal{S},L\log L}(f)|>\frac{t}{2\|b\|_{\rm BMO_{\triangle_\lambda}(\R_+)}}\})+ C_\phi\frac{ \|b\|_{\rm BMO_{\triangle_\lambda}(\R_+)}}{t} \int |f(x)|\cdot M_{\Phi\circ\phi}\left(\frac{w}{\mu}\right)(x)d\mu(x)\\
&\lesssim C_\phi \cdot \int \Phi\bigg(\|b\|_{\rm BMO_{\triangle_\lambda}(\R_+)}\cdot\frac{|f(x)|}{t}\bigg)M_{\Phi\circ\phi}\left(\frac{w}{\mu}\right)(x)d\mu(x).
\end{align*}
The proof is complete.
\end{proof}

\section{Proof of Theorem \ref{thmiff}}
\begin{itemize}
    \item $(\Longrightarrow)$: We show this direction first.
\end{itemize}
Note that if $\phi(t)\leq t^2$, for all $t\geq t_0$, then
\[
\Phi\circ\phi (t)\leq C\phi(t)\log(e+t),
\]
and hence if $\varepsilon\in(0,1)$ and $\phi(t):=t\log^{\varepsilon} (e+t)$, then we have 
\[
w(\{|[b,R_\lambda]f|>t\})\lesssim C_\phi  \int \Phi\bigg(2\|b\|_{\rm BMO_{\triangle_\lambda}(\R_+)}\cdot\frac{|f(x)|}{t}\bigg)M_{L(\log L)^{1+\varepsilon}}\left(\frac{w}{\mu}\right)(x)d\mu(x).
\]
Besides,\[
\log t\leq \frac{t^\alpha}{\alpha},\quad\forall t\geq 1,\alpha>0\implies M_{L(\log L)^{1+\varepsilon}}\left(\frac{w}{\mu}\right)(x)\leq \frac{c}{\alpha^{1+\varepsilon}} M_{L^{1+(1+\varepsilon)\alpha}}\left(\frac{w}{\mu}\right)(x).
\]
Since $1+(1+\varepsilon)\alpha>1$, then there is $1<p<\infty$ such that
\begin{align*}
M_{L(\log L)^{1+\varepsilon}}\left(\frac{w}{\mu}\right)(x)&\leq \frac{c}{\alpha^{1+\varepsilon}} M_{L^{1+(1+\varepsilon)\alpha}}\left(\frac{w}{\mu}\right)(x)\\
&\underset{p}{\lesssim} M\left(\frac{w}{\mu}\right)(x)\\
&\hskip-.35cm\underset{\widetilde{A}_{1,\lambda-\frac{1}{2}}}{\lesssim} w(x).
\end{align*}
Therefore,
\begin{align*}
    w(\{|[b,R_\lambda]f|>t\})&\lesssim C_\phi  \int \Phi\bigg(2\|b\|_{\rm BMO_{\triangle_\lambda}(\R_+)}\cdot\frac{|f(x)|}{t}\bigg)M_{L(\log L)^{1+\varepsilon}}\left(\frac{w}{\mu}\right)(x)d\mu(x)\\
    &\underset{w,\Phi}{\lesssim} \int \Phi\bigg(\|b\|_{\rm BMO_{\triangle_\lambda}(\R_+)}\frac{|f(x)|}{t}\bigg)w(x)dx.
\end{align*}
\begin{itemize}
    \item $\Longleftarrow$: Next, we prove this direction.
\end{itemize}

Suppose that
\[
\forall t>0,\quad 
 w(\{|[b,R_\lambda]f|>t\})\underset{w,\Phi}{\lesssim} \int \Phi\bigg(\hat{A}\frac{|f(x)|}{t}\bigg)w(x)dx.
\]
We will show that
\[
\exists A>0 \quad  s.t. \quad \forall B\implies w(\{x\in B: |b(x)-C_B|>A\})\leq \frac{1}{2}\cdot w(B).
\]
This will induce 
\[
b\in \rm
 BMO_{0,\frac{1}{2}}(\R_+,w).
\]
Since there exist positive constants $3\le A_1\le A_2$ such that for any interval $B:=B(x_0, r)\subset \R_+$, there exists another interval $\widetilde B:=B(y_0, r)\subset \R_+$
such that $A_1 r\le |x_0- y_0|\le A_2 r$, %$\frac12B\subset \widehat B\subset B$, %$\mu(B)\sim \mu(\widetilde B)$
and  for all $(x,y)\in ( B\times \widetilde{B})$, $K(x, y)$ does not change sign and
\[
|K(x, y)|\underset{A_1, A_2}{\gs} \frac1{\mu(\widetilde B)}.
\]
Let $\alpha_{b}(\widetilde B)$ be a median value of $b$ on the ball $\widetilde B $ with respect to $\mu$. Namely,  
$\alpha_{b}(\widetilde B)$ is a real number so that the two sets below have measures at least $ \frac{1}2 \mu  (\tilde B) $.  
\begin{align*}
\nonumber F_{+}\subset\{ y\in \tilde B: b(y)\geq \alpha_b(\tilde B) \}, \quad F_{-}\subset\{ y\in \tilde B: b(y)\leq \alpha_b(\tilde B)\}.
\end{align*}
Next, we define
$E_{+}=\{ x\in B: b(x)\geq \alpha_b(\tilde B) \}$ and $E_{-}=\{ x\in B: b(x)< \alpha_b(\tilde B) \}.$\ 
Then $B=E_{+}\cup E_{-}$ and $E_{+}\cap E_{-}=\emptyset$. 
And it is clear that 
\[\begin{split}\label{bx-by0-1 1}
b(x)-b(y) &\geq 0, \quad (x,y)\in E_{+}\times F_{-},\\ b(x)-b(y) &< 0, \quad (x,y)\in E_{-}\times F_{+}.
\end{split}\]
Let $B$ be an interval and take $C_B=\alpha_{b}(\widetilde B)$, then
\begin{align*}
    |b(x)-\alpha_{b}(\widetilde B)|\chi_B (x)&=(b(x)-\alpha_{b}(\widetilde B))\chi_{E_+} (x)+(\alpha_{b}(\widetilde B)-b(x))\chi_{E_-} (x)\\
    &\leq (b(x)-b(y))\chi_{E_+}(x)+(b(z)-b(x))\chi_{E_-}(x),
\end{align*}
for all $y\in F_{-}$ and $z\in F_+,$
and hence
\begin{align*}
 &|b(x)-\alpha_{b}(\widetilde B)|\chi_B (x)\\
 &\leq \frac{1}{\mu(F_-)}\int_{F_-} (b(x)-b(y))\chi_{E_+}(x)d\mu(y)+\frac{1}{\mu(F_+)}\int_{F_+} (b(z)-b(x))\chi_{E_-}(x)d\mu(z)\\
 &\lesssim \frac{1}{\mu(\widetilde B)}\int_{F_-} (b(x)-b(y))\chi_{E_+}(x)d\mu(y)+\frac{1}{\mu(\widetilde B)}\int_{F_+} (b(z)-b(x))\chi_{E_-}(x)d\mu(z)\\
 &\lesssim \int_{F_-} (b(x)-b(y))|K(x,y)|\chi_{E_+}(x)d\mu(y)+\int_{F_+} (b(z)-b(x))|K(x,y)|\chi_{E_-}(x)d\mu(z)\\
 &\leq \chi_{E_+}(x)\cdot |[b,R_\lambda](\chi_{F_-})|(x)+\chi_{E_-}(x)\cdot |[b,R_\lambda](\chi_{F_+})|(x).
\end{align*}
Therefore,
\begin{align*}
    &w(\{x\in B: |b(x)-C_B|>A\})\\
    &\leq w\left(\left\{x\in B: \chi_{E_+}(x)\cdot |[b,R_\lambda](\chi_{F_-})|(x)|>\frac{A}{2}\right\}\right)\\&\quad+w\left(\left\{x\in B: \chi_{E_-}(x)\cdot |[b,R_\lambda](\chi_{F_+})|(x)>\frac{A}{2}\right\}\right)\\
    &\lesssim\int_{E_+} \Phi\bigg(\hat{A}\frac{2}{A}\bigg)dw+\int_{E_-} \Phi\bigg(\hat{A}\frac{2}{A}\bigg)dw.
\end{align*}
There is $c(A_1,A_2)>0$ and take $A=\frac{\hat{A}}{c(A_1,A_2)}$, then
\[
 w(\{x\in B: |b(x)-C_B|>A\})\leq\frac{1}{2}w(B).
\]
To conclude, by Lemma \ref{emb} and Lemma \ref{sec},
\[
b\in {\rm BMO}_{0,\frac{1}{2}} (\mathbb{R_+ },w)\hookrightarrow{\rm BMO}_{L^p(w)} (\mathbb{R_+})\simeq\rm BMO_{\triangle_\lambda}.
\]

\section{Counterexample: Proof of Theorem \ref{thmcounter}}
Note that $\mu\in \widetilde{A}_{1,\lambda-\frac{1}{2}}$ and that the kernel of the Riesz operator has the following expression
\[
R_{\triangle_\lambda}(x,y)=-\frac{2\lambda}{\pi}\int^\infty_0 \frac{(x-y \cos \theta)(\sin\theta)^{2\lambda-1}}{(x^2+y^2-2xy\cos\theta)^{\lambda+1}}d\theta.
\]

First observe the following fact.
\begin{prop}
    $\log x^{2\lambda}\in\rm BMO_{\triangle_\lambda}(\R_+)$.
\end{prop}
\begin{proof}
Let $B:=(a,b)$ be an interval in $\R_+$.
 Then 
\begin{align*}
\frac{1}{b^{2\lambda+1}-a^{2\lambda+1}}\int^{b}_{a}\left|\log x^{2\lambda}-\log b^{2\lambda}\right|\cdot x^{2\lambda}dx
&=\frac{1}{b^{2\lambda+1}-a^{2\lambda+1}}\int^{b}_{a}2\lambda\left|\log (\frac{x}{b})\right|\cdot x^{2\lambda}dx\\
&=\frac{-2\lambda}{b^{2\lambda+1}-a^{2\lambda+1}}\int^{b}_{a}\log \frac{x}{b}\cdot x^{2\lambda}dx\\
&=\frac{-2\lambda}{b^{2\lambda+1}-a^{2\lambda+1}}\bigg(\log \frac{x}{b}\cdot\frac{x^{2\lambda+1}}{2\lambda+1}\bigg| ^b_a-\int^b_a\frac{1}{x}\frac{x^{2\lambda+1}}{2\lambda+1}dx\bigg)\\
&=\frac{-2\lambda}{b^{2\lambda+1}-a^{2\lambda+1}}\bigg(-\log \frac{a}{b}\cdot\frac{a^{2\lambda+1}}{2\lambda+1}-\frac{b^{2\lambda+1}-a^{2\lambda+1}}{(2\lambda+1)^2}\bigg),
\end{align*}
which implies that
\begin{align*}
 \frac{1}{b^{2\lambda+1}-a^{2\lambda+1}}\int^{b}_{a}\left|\log x^{2\lambda}-\log b^{2\lambda}\right|\cdot x^{2\lambda}dx
 &= \frac{2\lambda}{b^{2\lambda+1}-a^{2\lambda+1}}\cdot\log \frac{a}{b}\cdot\frac{a^{2\lambda+1}}{2\lambda+1}+\frac{2\lambda}{(2\lambda+1)^2}  \\
 &\underset{\lambda}{\lesssim}\frac{\log \frac{a}{b}}{(\frac{b}{a})^{2\lambda+1}-1}+1.
\end{align*}
Since
\[
\lim_{t\to0^+}t^{2\lambda+1}\log t=0,
\]
then we have
\[
\log x^{2\lambda}\in {\rm BMO_{\triangle_\lambda}(\R_+)}.
\]
The proof is complete.
\end{proof}

\medskip
\begin{itemize}
    \item \noindent {\bf Counterexample:}
\end{itemize}

Consider the function $f$ to be a Dirac mass at point $\varepsilon\ll 1 $, then for $x>e^{\frac{1}{2\lambda+2}}$ the commutator
\[
[b,R_\lambda]f(x)\sim \log x\cdot \left(\frac{1}{x}\right)^{2\lambda+1}\cdot \varepsilon^{2\lambda}-\left(\frac{1}{x}\right)^{2\lambda+1}\cdot\log (\varepsilon) \cdot\varepsilon^{2\lambda}.
\]
Then define \[g(x):= \log x\cdot \left(\frac{1}{x}\right)^{2\lambda+1}\cdot \varepsilon^{2\lambda}-\left(\frac{1}{x}\right)^{2\lambda+1}\cdot\log (\varepsilon) \cdot\varepsilon^{2\lambda},\]
which is decreasing when $x>e^{\frac{1}{2\lambda+2}}$, and hence
\begin{align*}
    t \cdot\mu(\{x\in\R_+ :|[b,R_\lambda]f(x)|>t\})&\geq t \cdot\mu(\{x>e^{\frac{1}{2\lambda+2}} :|[b,R_\lambda]f(x)|>t\}) \\
    &\sim t\bigg((g^{-1}(t))^{2\lambda+1}-1\bigg).
\end{align*}
Since
\begin{align*}
    \lim_{t\to 0} t\cdot \left(g^{-1}(t)\right)^{2\lambda+1}=\lim_{t\to \infty}g(t)\cdot t^{2\lambda+1}= \infty,
\end{align*}
then
\[
 t |\{x\in\R_+ :|[b,R_\lambda]f(x)|>t\}|\underset{t\to 0^+}{\longrightarrow}\infty.
\]

\bigskip
\bigskip

\noindent {\bf Acknowledgement:} J. Li and B. D. Wick are supported by ARC DP 220100285. B. D. Wick is partially supported by NSF-DMS \# 2000510, \# 2054863, and \# 1800057. C.-W. Liang is supported by MQ --NTU Cotutelle PhD Scholarship. C.-Y.
Shen was supported in part by NSTC through grant 111-2115-M-002-010-MY5.

\bigskip 

\printbibliography[heading=bibintoc,title={References}]

(J. Li) Department of Mathematics, Macquarie University, NSW, 2109, Australia.\\ 
{\it E-mail}: \texttt{ji.li@mq.edu.au}\\

(C.-W. Liang) Department of Mathematics, National Taiwan University, Taiwan.
\\ 
{\it E-mail}: \texttt{d10221001@ntu.edu.tw}\\

(C.-Y. Shen) Department of Mathematics, National Taiwan University, Taiwan. \\ {\it E-mail}: \texttt{cyshen@math.ntu.edu.tw}\\

(B.D. Wick) Department of Mathematics
Washington University - St. Louis
. \\ {\it E-mail}: \texttt{wick@math.wustl.edu}

\vspace{0.3cm}

\end{document}